\documentclass[12pt]{amsart}
\usepackage{amsmath}
\usepackage{amsthm}
\usepackage{amssymb}
\usepackage{amscd}
\usepackage{amsfonts}
\usepackage{amsbsy}
\usepackage{epsfig,afterpage}
\usepackage[dvips]{psfrag}
\usepackage{subfigure}
\usepackage{srcltx}

\usepackage{verbatim}
\usepackage{graphicx}

\usepackage{multirow}
\usepackage{graphics}
\usepackage{epstopdf}
\usepackage{overpic}
\usepackage{multicol}

\usepackage[colorlinks=true, linkcolor=blue, citecolor=red, urlcolor=blue]{hyperref}

\newtheorem {theorem} {Theorem}
\newtheorem {proposition} [theorem]{Proposition}

\newtheorem {lemma}  [theorem]{Lemma}

\newtheorem {remark} [theorem]{Remark}

\newtheorem{mtheorem}{Theorem}

\usepackage{float}
\usepackage{tikz, wrapfig}
\tikzset{node distance=3cm, auto}
\allowdisplaybreaks

\begin{document}

\title[Canard cycles and non-linear regularizations]{
Canard cycles of non-linearly regularized piecewise smooth vector fields}

\author[P. De Maesschalck, R. Huzak and O.H. Perez]
{Peter De Maesschalck$^{1}$, Renato Huzak$^{1}$ 
and Otavio Henrique Perez$^{2,1}$}

\address{$^{1}$Hasselt University, Campus Diepenbeek, Agoralaan Gebouw D, 3590 Diepenbeek, Belgium}

\address{$^{2}$Universidade de S\~{a}o Paulo (USP), Instituto de Ci\^{e}ncias Matem\'aticas e de Computa\c{c}\~{a}o (ICMC). Avenida Trabalhador S\~{a}o Carlense, 400, CEP 13566-590, S\~{a}o Carlos, S\~{a}o Paulo, Brazil.}

\email{peter.demaesschalck@uhasselt.be}
\email{renato.huzak@uhasselt.be}
\email{otavio.perez@icmc.usp.br}

\thanks{ .}


\date{}
\dedicatory{}

\begin{abstract}
The main purpose of this paper is to study limit cycles in non-linear regularizations of planar piecewise smooth systems with fold points (or more degenerate tangency points) and crossing regions. We deal with a slow fast Hopf point after non-linear regularization and blow-up. We give a simple criterion for upper bounds and the existence of limit cycles of canard type, expressed in terms of zeros of the slow divergence integral. Using the criterion we can construct a quadratic regularization of piecewise linear center such that for any integer $k>0$ it has at least $k+1$ limit cycles, for a suitably chosen monotonic transition function $\varphi_k:\mathbb{R}\rightarrow\mathbb{R}$. We prove a similar result for regularized invisible-invisible fold-fold singularities of type II$_2$. Canard cycles of dodging layer are also considered, and we prove that such limit cycles undergo a saddle-node bifurcation.

\end{abstract}

\maketitle

\section{Introduction and statement of the problem}

In the framework of planar piecewise smooth vector fields (PSVF for short), the fold-fold singularity is known for its rich and interesting bifurcation diagram. Such a singularity can be classified into three cases: visible-visible (VV), visible-invisible (VI) and invisible-invisible (II). Each case has subcases that must be considered depending on the \emph{sliding} and \emph{crossing} regions of the switching locus $\Sigma$. See Figure \ref{fig-fold-fold} in Section \ref{section-continuous-comb} and \cite[Section 3.2]{KGR} for further details on bifurcations of such singularity.

It is known that a generic fold-fold singularity has codimension 1 (see, for instance, \cite[Subsection 4.1.1]{GST}). In fact, in the Case II, one of the possible subcases is called \emph{non-smooth focus} or \emph{pseudo focus} (which is also denoted by II$_{2}$ in the above references), since in a small neighborhood of this singularity the switching locus has only crossing points and it presents a focus-like behavior. As discussed in \cite[Sections 4 and 7]{GST}, depending on the coefficients of the Taylor series of the \emph{first return map}, the II$_{2}$ fold-fold singularity can have codimension equal or higher than 1. In the $\operatorname{cod}1$ case, there is one \emph{crossing limit cycle} bifurcating from it. 

In papers \cite{BLS, KH} the authors provided a systematic analysis of regularizations of PSVFs having a (generic) fold-fold singularity positioned at the origin. Concerning the II$_{2}$ case, the authors proved the existence of two limit cycles of the regularized system for a suitable \emph{transition function} $\varphi$ and region of the parameter space $(\varepsilon,\mu)$, in which $\varepsilon$ and $\mu$ stand for the singular perturbation and bifurcation parameters, respectively. It is important to remark that the results obtained in \cite{BLS, KH} concerned the Sotomayor--Teixeira regularization \cite{SotoTeixeira} (ST regularization for short). See also Section \ref{section-definitions}.

In \cite{SotoMachado} the authors studied conditions that a PSVF $Z$ must satisfy so that its ST regularization is structurally stable. Concerning limit cycles, in \cite[Proposition 13]{SotoMachado} the authors proved that hyperbolic crossing limit cycles of $Z$ persist after ST regularization (for $\varepsilon > 0$ sufficiently small). One of the limit cycles of the ST-regularized II$_{2}$ fold-fold obtained in \cite{BLS, KH} is related to the crossing one of $Z$. The second limit cycle obtained is located inside the regularization stripe and its existence depends on the transition function adopted.

Even though the ST regularization is widely used in both applied and theoretical problems, it is quite natural to ask what happens to the number of limit cycles if one applies other regularization processes. In this paper, our goal is to study planar PSVFs in the presence of fold points (or more degenerate tangency points) and crossing regions via \emph{non-linear regularizations} \cite{NovaesJeffrey, SilvaSarmientoNovaes}. See Section \ref{section-regularization} for a precise definition.

Non-linear regularizations can produce different phenomena which the ST regularization cannot (see also \cite{PRS}), and one of our goals is to generate more limit cycles than those $2$ obtained in \cite{BLS, KH} (see Theorems \ref{thm-center} and \ref{thm-ii2} in Subsection \ref{section-state-resul}). This paper can also be seen as continuation of \cite{HuzakKristiansen}, in which the authors applied slow divergence integrals to study the number of limit cycles of the regularized fold-fold VI$_{3}$. We highlight that in previous papers such tool was used near fold-fold singularities having sliding regions. To the best of our knowledge, this is the first time that the slow divergence integral is used near fold-fold singularities having crossing regions.

Since the notion of slow divergence integral plays a key role in this paper, we will first explain it for planar smooth slow-fast systems \cite[Chapter 5]{DMDR} and regularized planar piecewise smooth systems with sliding \cite{HuzakKristiansen,HuzakKristiansen2,HuzakKristiansen3}. 
In the planar slow-fast setting and canard theory, the slow divergence integral was developed by De Maesschalck, Dumortier and Roussarie (see \cite{DeMaesschalckDumortier1,DMDR,DeMaesschalckHuzak,DPR,DR1996} and references therein). Consider, for example, the following smooth system with a slow-fast Hopf point at the origin $(x,y)=(0,0)$
\begin{align}\label{SDI-example-intro}
  \begin{cases} 
  \dot x=-xy+\varepsilon\left(\alpha-y+y^2f(y)\right),\\
  \dot y=x ,\\ 
  \end{cases}
  \end{align}
where $\varepsilon\ge 0$ is the singular perturbation parameter kept small, $\alpha\in \mathbb R$ is close to $0$ and $f$ is a smooth function. In this paper, by ``smooth" we mean $C^\infty$-smooth. If $\varepsilon=0$ in \eqref{SDI-example-intro}, then we deal with the fast subsystem (often called the fast dynamics). The fast subsystem has the curve of singularities $\{x=0\}$ which is normally attracting when $y>0$ (the nonzero eigenvalue is negative) and normally repelling when $y<0$ (the nonzero eigenvalue is positive). These two branches of the line of singularities are separated by a singularity of nilpotent type at $(0,0)$. The fast fibers are parabolas $x=-\frac{1}{2}y^2+C$ and the line of singularities has a (quadratic) contact with the fibers at the origin (sometimes we call the origin a contact point).  We refer to Figure \ref{fig-example-intro}.

\begin{figure}[ht]\center{
\begin{overpic}[width=0.2\textwidth]{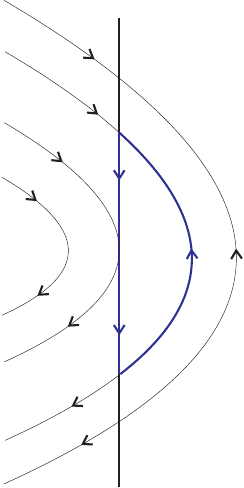}
\put(31,48){$\Gamma_{y}$}
\end{overpic}}
\caption{\footnotesize{Phase portrait of system \eqref{SDI-example-intro}. The canard cycle $\Gamma_{y}$ is highlighted in blue.}}
\label{fig-example-intro}
\end{figure}

An important observation is that, near normally hyperbolic points $y\ne 0$, there exist invariant manifolds of \eqref{SDI-example-intro} asymptotic to the line $\{x=0\}$ (they correspond to center manifolds if one augments system \eqref{SDI-example-intro} with $\dot \varepsilon=0$). Using standard asymptotic expansions in $\varepsilon$ (see e.g. \cite{DMDR}) we obtain
$$x=\varepsilon\left(\frac{\alpha-y+y^2f(y)}{y}+O(\varepsilon)\right).$$

If we now substitute this for $x$ in the $y$-component of \eqref{SDI-example-intro}, divide out $\varepsilon$ and let $\varepsilon$ tend to $0$, then we get the slow dynamics \cite[Chapter 3]{DMDR} 
$$y'=\frac{\alpha-y+y^2f(y)}{y}, \ y\ne 0.$$

Notice that for $\alpha=0$ the slow dynamics has a removable singularity in $y=0$ and it is regular there ($y'=-1+yf(y)$). It is also clear that the slow dynamics points (at least near $y=0$) from the normally attracting branch $y>0$ to the normally repelling branch $y<0$. This produces so-called canard trajectories of \eqref{SDI-example-intro} which follow the attracting branch, pass through the contact point and then stay close to the repelling branch for some time.  

We define now the notion of slow divergence integral \cite[Chapter 5]{DMDR}. Suppose that the slow dynamics has no singularities. The slow divergence integral computed along the slow segment $[-y,y]\subset \{x=0\}$ for $\alpha=0$ is given by
$$I(y)=\int_{-y}^y\frac{sds}{-1+sf(s)}, \ y>0.$$

This is an integral of the divergence of \eqref{SDI-example-intro} for $\varepsilon=0$ ($-y$), with respect to
the slow time denoted by $\tau$ ($d\tau=\frac{dy}{-1+yf(y)}$). It is well-known that zeros of the function $ I$ provide candidates for limit cycles of \eqref{SDI-example-intro}, produced by canard cycles. More precisely, for a fixed $y>0$ and $(\varepsilon,\alpha)=(0,0)$, the canard cycle $\Gamma_y$ consists of the segment $[-y,y]$ and the fast orbit connecting $(0,-y)$ and $(0,y)$ (see Figure \ref{fig-example-intro}). When $I$ has a zero of multiplicity $k\ge 1$ at $y=y_0$, the canard cycle $\Gamma_{y_0}$ can generate at most $k+1$ limit cycles for $(\varepsilon,\alpha)$ close to $(0,0)$ (see \cite{DeMaesschalckDumortier3}).  
\smallskip

The notion of slow divergence integral in a regularized piecewise smooth VI$_3$ model was introduced in \cite{HuzakKristiansen,HuzakKristiansen2}. Near the VI$_3$ fold-fold singularity \cite{HuzakKristiansen,KGR}, the Filippov sliding vector field \cite{Filippov} points from the stable sliding region
to the unstable sliding region with non-zero speed (see Figure \ref{fig-VI3-intro}(a)).

\begin{figure}[ht]\center{
\begin{overpic}[width=0.55\textwidth]{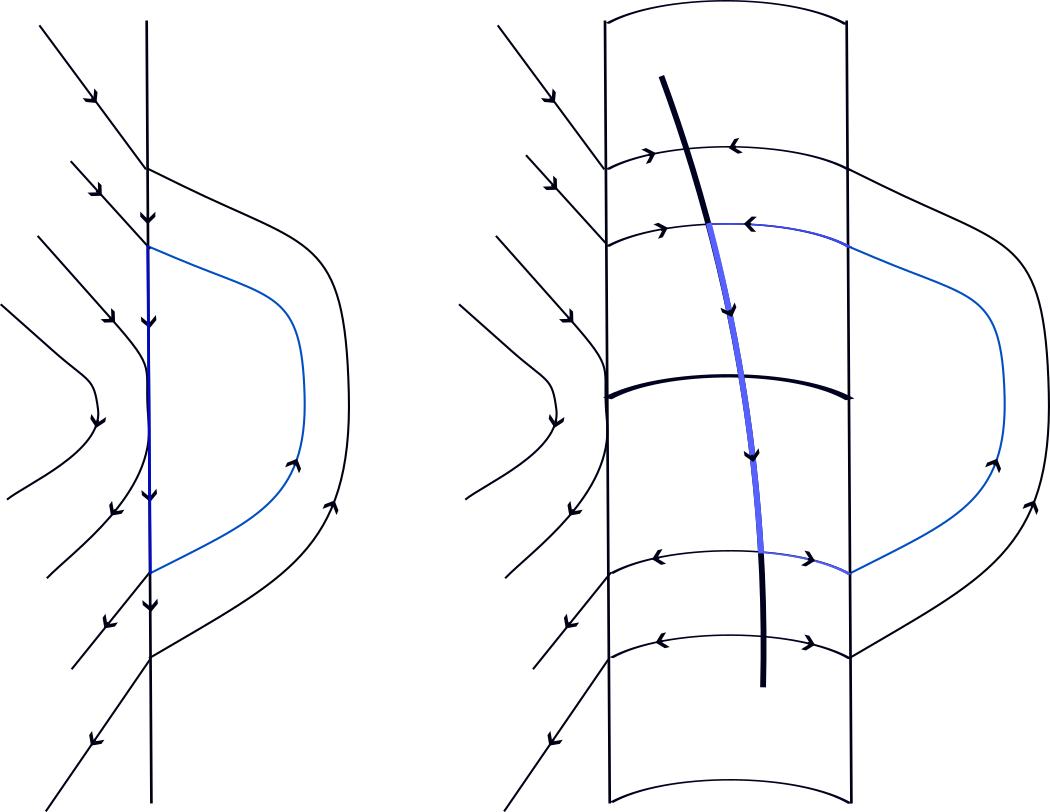}
\put(23,36){$\widehat\Gamma$}
\put(11,-4){(a)}
\put(89,36){$\widehat\Gamma$}
\put(66,10){$C_0$}
\put(62,37){$C_1$}
\put(65,-4){(b)}
\end{overpic}}
\caption{\footnotesize{(a) The VI$_3$ fold-fold singularity (the folds have opposite visibility) where $\{x=0\}$ is the switching line and the Filippov sliding vector field points from the stable sliding region $y>0$
to the unstable sliding region $y<0$ with nonzero speed. $\widehat\Gamma$ is called a sliding cycle. (b) Dynamics of a regularized VI$_3$ model on the blow-up cylinder. $C_0$ and $C_1$ are curves of singular points.}} 
\label{fig-VI3-intro}
\end{figure}
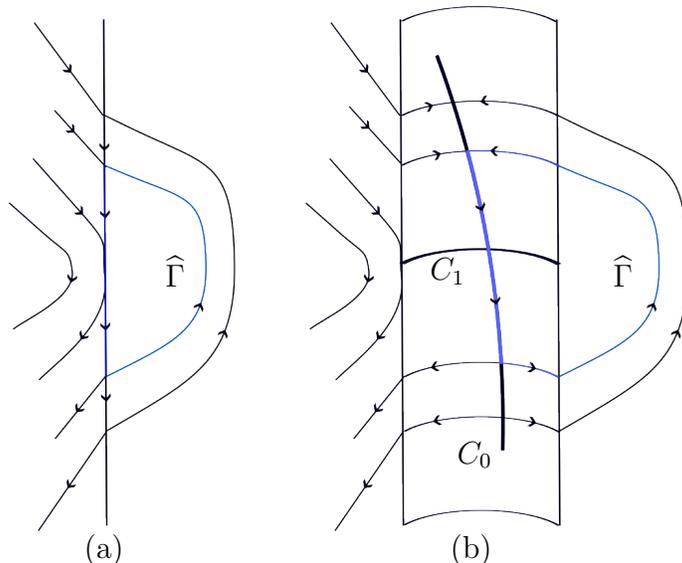

Notice that the graphic $\widehat\Gamma$ in Figure \ref{fig-VI3-intro}(a), consisting of an orbit located in the half-plane
$x\ge 0$ with invisible fold point and the portion of $\{x=0\}$ connecting the end points of that orbit, is similar to the canard cycle $\Gamma_y$ in Figure \ref{fig-example-intro} (the switching line with the Filippov sliding vector field defined on it plays the
role of the curve of singularities of \eqref{SDI-example-intro} with the associated slow dynamics). We call $\widehat\Gamma$ a sliding cycle.

Following \cite{HuzakKristiansen,HuzakKristiansen2}, we can connect $\widehat\Gamma$ with the notion of slow divergence integral using a linear regularization of the VI$_3$ fold-fold singularity. The regularized system becomes a slow-fast system upon a suitable cylindrical blow-up of the switching line, and then we can compute the slow divergence integral along the slow segment of the blown up sliding cycle $\widehat\Gamma$ (see Figure \ref{fig-VI3-intro}(b)). The curve of singularities $C_0$ of the slow-fast system defined on the
blow-up cylinder is normally hyperbolic away from the intersection with $C_1$ and the flow of the Filippov sliding vector field is the slow dynamics along $C_0$ (regularly extended through the intersection). In \cite[Theorem 3.1]{HuzakKristiansen}, one can find a criterion for the existence of limit cycles of the linearly regularized system produced by sliding cycles. The criterion is expressed in terms of simple zeros of the slow divergence integral. For more details, we refer the reader to \cite{HuzakKristiansen,HuzakKristiansen2}. 

The main purpose of \cite{HuzakKristiansen3}, which is a natural continuation of \cite{HuzakKristiansen}, was to introduce the notion of slow divergence integral for the other fold-fold singularities of sliding type VV$_1$, VI$_2$, II$_1$ (see \cite{KGR} or \cite[Figure 2.2]{HuzakKristiansen3}), one-sided tangency points with sliding, etc. 

The papers \cite{HuzakKristiansen,HuzakKristiansen2,HuzakKristiansen3} deal with fold-fold singularities of \emph{sliding type}. In this paper, we use the slow divergence integral to study limit cycles in regularized fold-fold singularities of \emph{crossing type}, and we focus our analysis in the II$_2$ (which is a generic singularity, see \cite{KGR}) and non-smooth center (which is non generic, see \cite{BuzziCarvalhoTeixeira} and Figure \ref{fig-intro-crossing}(a)) cases. One can also expect that the ideas used in this paper can also be used to study the singularities VV$_2$ and VI$_1$ (see \cite{KGR}).

It is important to remark that this paper concerns limit cycles of the (non-linearly) regularized PSVF, which is $C^{\infty}$-smooth. The limit cycles inside the regularization stripe shrink to the switching manifold as $\varepsilon\rightarrow 0$. However, we show that non-linear regularizations produce more limit cycles than the ST regularization, even if the PSVF $Z$ is piecewise linear. Indeed, it is shown in Theorems \ref{thm-center} and \ref{thm-ii2} below that piecewise linear vector fields can produce $k$ limit cycles under non-linear regularization, for suitable monotonic transition function $\varphi_{k}$. Observe that a similar result was proved in \cite{HuzakKristiansen} for linear regularization, but the PSVF considered was quadratic.

We deal with non-linear regularizations because one of our goals is to generate more limit cycles in regularizations of invisible-invisible fold-fold singularities than \cite{BLS, KH}. See Theorems \ref{thm-center} and \ref{thm-ii2} below. With this framework, the non-linearly regularized vector field presents a slow-fast Hopf point of the associated slow-fast system (see Section \ref{subsec-hopf}). Observe that in \cite{PRS} the authors proved that, by dropping the monotonicity condition in the Sotomayor--Teixeira regularization, it is possible to generate a (planar) slow-fast jump point (often called SF-generic fold). 

\begin{figure}[ht]\center{
\begin{overpic}[width=0.8\textwidth]{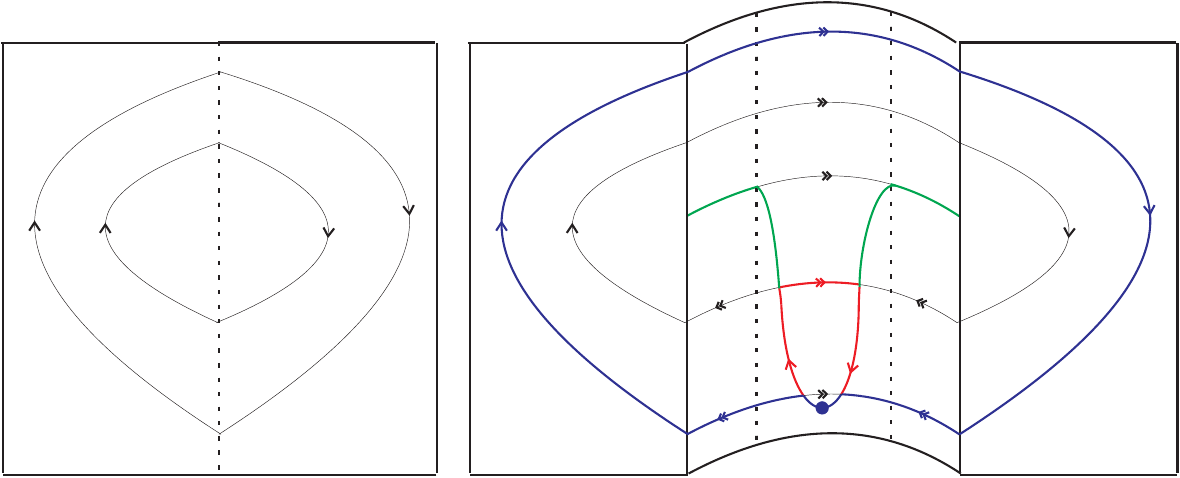}
\put(17,-3){(a)}
\put(67,-3){(b)}
\put(70,22){$C_0$}
\end{overpic}}
\caption{\footnotesize{(a) A nonsmooth center (Case II can present a center-like behavior, see \cite{BuzziCarvalhoTeixeira}). (b) Dynamics on the
blow-up cylinder using a non-linear regularization of the nonsmooth center. The curve of singularities $C_0$ contains a normally attracting branch, a normally repelling branch and a slow-fast Hopf point between them. Two different types of canard cycles (red and blue).}}
\label{fig-intro-crossing}
\end{figure}

The presence of a slow-fast Hopf point (after non-linear regularization and blow-up) plays an important role. The slow-fast Hopf point can generate limit cycles which may grow and become Hausdorff close to canard cycles (see the red graphic in Figure \ref{fig-intro-crossing}(b)). As the size of the canard cycle increases, one can expect limit cycles bifurcating from the blue canard cycle in Figure \ref{fig-intro-crossing}(b). The goal of our paper is to study the number of limit cycles produced by these two types of canard cycles. See Theorem \ref{theo-main} in Section \ref{sec-terminal-case}.

Theorem \ref{theo-main} is a crucial result for prove Theorems \ref{thm-center} and \ref{thm-ii2}. More precisely, we give a simple criterion (expressed in terms of the slow divergence integral) for upper bounds and the existence of limit cycles produced by the blue canard cycle. Roughly speaking, if the slow divergence integral has a zero of multiplicity $k$, then the canard cycle can produce at most $k+1$ limit cycles (see Theorem \ref{theo-main}(a)). Simple zeros of the slow divergence integral correspond to hyperbolic limit cycles for unbroken breaking parameter (see Theorem \ref{theo-main}(b)). Theorem \ref{theo-main} is stated and proven in a more general setting with two folds (not necessarily of invisible type), two singularities as tangency points or their combination (see Subsection \ref{subsec-fold-points}). We refer to Figure \ref{fig-limit-cycles}(a) in Section \ref{sec-terminal-case}. A similar criterion has been proven for the red canard cycle in smooth planar slow-fast systems (see e.g. \cite{Dumortier}).

Theorem \ref{theo-dodging} in Section \ref{section-dodging-case} deals with canard cycles in case of a dodging layer, given in Figure \ref{fig-limit-cycles}(b). Such canard cycles can produce at most $2$ limit cycles and a simple zero of the slow divergence integral corresponds to a codimension $1$ bifurcation of limit cycles (saddle-node bifurcation). To the best of our knowledge, this is the first time that canard cycles of dodging layer appears in regularized piecewise smooth vector field. In the smooth setting, canard cycles with a dodging layer have been studied in \cite{DMDR,JJR1,JJRS}.


This paper is structured as follows. In Section \ref{section-definitions} we present the main tools, such as piecewise smooth vector fields, regularizations and we state Theorems \ref{thm-center} and \ref{thm-ii2}. Section \ref{section-model-def} is devoted to define assumptions in the model that will be studied throughout this paper. We precisely define the types of limit periodic sets that we are interested in Section \ref{sec-candidates}. Theorems \ref{thm-center}, \ref{thm-ii2} and \ref{theo-main} are proven in Section \ref{sec-terminal-case} and in Section \ref{section-dodging-case} we prove Theorem \ref{theo-dodging}. In this paper, continuous combinations, transition functions, etc., are ($C^\infty$-)smooth. The main reason is that we often refer to \cite{DMDR} where planar slow-fast systems have been studied in the smooth setting.



\section{Preliminary definitions and results}\label{section-definitions}

Piecewise smooth vector fields are widely adopted to model phenomena of many branches of applied sciences \cite{bernardo,Filippov}. A closed set with empty interior $\Sigma$ divides the phase space in finitely many open sets, and on each open set is defined a smooth vector field. In this paper we suppose that the straight line $\Sigma = \{x = 0\}$ divides the phase plane in two open regions, and the smooth vector fields $X$ and $Y$ are defined in the regions $\{x > 0\}$ and $\{x < 0\}$, respectively. In what follows we precisely define this framework using \emph{non-linear regularizations} \cite{NovaesJeffrey, SilvaSarmientoNovaes}.

\subsection{Continuous combinations and piecewise smooth vector fields}\label{section-continuous-comb}
A \emph{continuous combination} is a vector field depending on a parameter $\lambda$ written as
\begin{equation*}
\widetilde{Z}_{\mu}(\lambda,x,y) = \big{(}\widetilde{Z}_{1,\mu}(\lambda,x,y), \widetilde{Z}_{2,\mu}(\lambda,x,y)   \big{)},    
\end{equation*}
with $\lambda \in \mathbb{R}$, $(x,y)\in U\subset\mathbb{R}^{2}$, $U$ is an open set and $\mu\in\mathbb{R}^{l}$ denotes finitely many parameters $\mu = (\mu_{1},\dots,\mu_{l})$. Although we assume  $\widetilde{Z}$ to be smooth with respect to $(\lambda,x,y,\mu)$, we will keep the terminology \emph{continuous} for two reasons. Firstly, we want to be in accordance with the terminology adopted in \cite{NovaesJeffrey,SilvaSarmientoNovaes}. Secondly, being \emph{continuous} extends the notion of being \emph{convex} in a sense that we will precise soon.

Define the smooth vector fields
\begin{equation*}
\begin{array}{ccccc}
X_{\mu}(x,y) & = & \big{(}X_{1,\mu}(x,y), X_{2,\mu}(x,y)\big{)} & := & \widetilde{Z}_{\mu}(1,x,y), \\
Y_{\mu}(x,y) & = & \big{(}Y_{1,\mu}(x,y), Y_{2,\mu}(x,y)\big{)} & := & \widetilde{Z}_{\mu}(-1,x,y).
\end{array}
\end{equation*}

A piecewise smooth vector field (PSVF for short) is defined as
\begin{equation*}
Z_{\mu}(x,y) = \widetilde{Z}_{\mu}\Big{(}\operatorname{sgn}\big{(}F(x,y)\big{)},x,y\Big{)},
\end{equation*}
in which $F:U\subset\mathbb{R}^{2}\rightarrow \mathbb{R}$. It is straightforward to verify that $Z_{\mu}(x,y) = X_{\mu}(x,y)$ on $\{F(x,y) > 0\}$ and $Z_{\mu}(x,y) = Y_{\mu}(x,y)$ on $\{F(x,y) < 0\}$. One may also denote $Z_{\mu} = (X_{\mu},Y_{\mu})$ in order to stress the dependency of $Z_{\mu}$ on the smooth vector fields $X_{\mu}$ and $Y_{\mu}$. The set $\Sigma = \{F(x,y) = 0\}$ is called \emph{switching set}, and in the case where $\Sigma$ is a manifold we call it \emph{switching manifold}.

Conversely, any given PSVF gives rise to many associated combinations. A well-known example of continuous combination is
\begin{equation}\label{eq-convex-combination}
\widetilde{Z}_{\mu}(\lambda,x,y) = \frac{1 + \lambda}{2}X_{\mu}(x,y) + \frac{1 - \lambda}{2}Y_{\mu}(x,y),
\end{equation}
which was called \emph{convex combination} in \cite{NovaesJeffrey, SilvaSarmientoNovaes}. Replacing $\lambda$ by $\operatorname{sgn}\big{(}F(x,y)\big{)}$ in equation \eqref{eq-convex-combination}, one obtains a PSVF as studied in \cite{GST}. In short, the idea of the definition of continuous combination is to allow any curve to connect the points $X_{\mu}(p)$ and $Y_{\mu}(p)$, whereas in the convex case a line segment connects such points. See Figure \ref{fig-continuous-comb}.

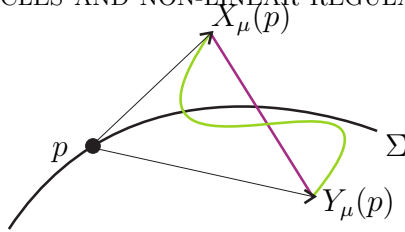
\begin{figure}[ht]\center{
\begin{overpic}[width=0.3\textwidth]{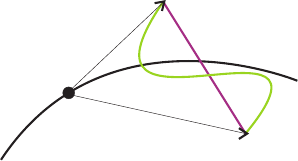}
\put(12,20){$p$}
\put(55,55){$X_{\mu}(p)$}
\put(85,5){$Y_{\mu}(p)$}
\put(102,20){$\Sigma$}
\end{overpic}}
\caption{\footnotesize{Convex and continuous combinations of $X_{\mu}$ and $Y_{\mu}$.}}
\label{fig-continuous-comb}
\end{figure}

The \emph{Lie-derivative} of $F$ with respect to the vector field $X_{\mu}$ is given by $X_{\mu}F = \langle X_{\mu},\nabla F \rangle$ and $X_{\mu}^{i}F = \langle X_{\mu},\nabla X_{\mu}^{i-1}F \rangle$ for all integers $i\geq 2$. This allows us to define the following regions in $\Sigma$:
\begin{enumerate}
  \item \emph{Sewing region}: $\Sigma^{w} = \big{\{}(x,y)\in\Sigma \ ; \ X_{\mu}F  Y_{\mu}F  > 0\big{\}}$,
  \item \emph{Sliding region}: $\Sigma^{s} = \big{\{}(x,y)\in\Sigma\ ; \  X_{\mu}F  Y_{\mu}F  < 0\big{\}}$.
\end{enumerate}

Following Filippov's convention \cite{Filippov}, one can define a vector field in $\Sigma^{s}\subset\Sigma$. The \emph{Filippov sliding vector field} associated to $Z_{\mu}$ is the vector field $Z_{\mu}^{\Sigma}:\Sigma^s\rightarrow T\Sigma$ given by
\begin{equation*}
Z_{\mu}^{\Sigma}(x,y) = \displaystyle\frac{1}{Y_{\mu}F - X_{\mu}F}\Big{(}X_{\mu}  Y_{\mu}F - Y_{\mu}  X_{\mu}F\Big{)}.
\end{equation*}

A point $\mathbf{p}_{0}\in \Sigma$ is called a \textit{tangency point} if $X_{\mu}F(\mathbf{p}_{0}) = 0$ or $Y_{\mu}F(\mathbf{p}_{0}) = 0$.
A point $\mathbf{p}_{0}\in \Sigma$ is a \textit{fold point of $X_{\mu}$} if $X_{\mu}F(\mathbf{p}_{0}) = 0$ and $X_{\mu}^{2}F(\mathbf{p}_{0})\neq 0$. If $X_{\mu}^{2}F(\mathbf{p}_{0}) > 0$, $\mathbf{p}_{0}$ is a \textit{visible fold of $X_{\mu}$} and if $X_{\mu}^{2}F(\mathbf{p}_{0}) < 0$ we say that $\mathbf{p}_{0}$ is an \textit{invisible fold of $X$}. Fold points of $Y_{\mu}$ are defined in an analogous way, however $\mathbf{p}_{0}$ is visible if $Y_{\mu}^{2}F(\mathbf{p}_{0}) < 0$ and invisible if $Y_{\mu}^{2}F(\mathbf{p}_{0}) > 0$. If $\mathbf{p}_{0}$ is a fold of both $X_{\mu}$ and $Y_{\mu}$ simultaneously, then $\mathbf{p}_{0}$ is called \textit{fold-fold singularity of $Z_{\mu}$}. This singularity can be classified into three types: visible-visible, visible-invisible and invisible-invisible. See Figure \ref{fig-fold-fold}.

\begin{figure}[ht]\center{
\begin{overpic}[width=0.4\textwidth]{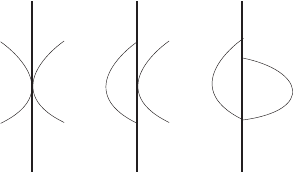}
\put(10,61){$\Sigma$}
\put(46,61){$\Sigma$}
\put(81,61){$\Sigma$}
\end{overpic}}
\caption{\footnotesize{Fold-fold singularities. From the left to the right: visible-visible, visible-invisible and invisible-invisible.}}
\label{fig-fold-fold}
\end{figure}

\subsection{Regularizations of piecewise smooth vector fields}\label{section-regularization}

We say that $\varphi:\mathbb{R}\rightarrow\mathbb{R}$ is a \emph{transition function} if the following conditions are satisfied: \textbf{(1)} $\varphi$ is smooth; \textbf{(2)} $\varphi(t) = -1$ if $t \leq -1$ and $\varphi(t) = 1$ if $t \geq 1$. The transition function is \emph{monotonic} if it satisfies
\textbf{(3)} $\varphi'(t) > 0$ if $s\in (-1,1)$.


Let $\widetilde{Z}_{\mu}$ be a continuous combination associated with the PSVF $Z_{\mu}$.  A \emph{$\varphi$ non-linear regularization of $Z_{\mu}$} is the $(\varepsilon,\mu)$-family of smooth vector fields given by
\begin{equation}\label{eq-def-regularization}
\widetilde{Z}^{\varphi}_{\varepsilon,\mu}(x,y):= \widetilde{Z}_{\mu}\Big{(}\varphi\Big{(}\frac{F(x,y)}{\varepsilon}\Big{)},x,y\Big{)}.    
\end{equation} 

If $\widetilde{Z}_{\mu}$ is a convex combination, we say that $\widetilde{Z}^{\varphi}_{\varepsilon,\mu}$ is a \emph{$\varphi$-linear regularization of $Z_{\mu}$}. When $\widetilde{Z}_{\mu}$ is convex and $\varphi$ is monotonic, \eqref{eq-def-regularization} is the well-known \emph{Sotomayor--Teixeira regularization} \cite{SotoTeixeira}. We keep the superscript $\varphi$ in the notation $\widetilde{Z}^{\varphi}_{\varepsilon,\mu}$ to emphasize the dependency of the regularization on the transition function. Indeed, different transition functions lead to different dynamics of the regularized vector field (see \cite{HuzakKristiansen, PanazzoloSilva, PRS}).

In \cite[Theorem 1]{NovaesJeffrey} the authors proved the following result. \emph{Let $\varphi$ and $\psi$ be monotonic and non-monotonic transition functions, respectively. If $\widetilde{Z}^{\psi}_{\varepsilon,\mu}$ is a $\psi$-linear regularization, then there exists a unique non-linear regularization $\widetilde{W}^{\varphi}_{\varepsilon,\mu}$ such that $\widetilde{Z}^{\psi}_{\varepsilon,\mu} = \widetilde{W}^{\varphi}_{\varepsilon,\mu}$}. However, in general the converse is not true (see \cite[Theorems B and C]{PRS}).

From now on, we consider a coordinate system such that $F(x,y) = x$, therefore $\Sigma = \{x = 0\}$. Moreover, we deal with non-linear regularizations in order to obtain richer phenomena, namely \emph{slow-fast Hopf points} (and generic turning points). Indeed, such singularities do not appear in linear regularizations, even if we consider non monotonic transition functions (see Section \ref{subsec-hopf}).

\subsection{Statement of results}\label{section-state-resul}

We are interested in limit cycles of non-linear regularizations as in Figure \ref{fig-intro-crossing}. Such limit cycles can be ``large'' or ``small'', as highlighted in blue or red in Figure \ref{fig-intro-crossing}, respectively. 

We emphasize that the PSVF's in Theorems \ref{thm-center} and \ref{thm-ii2} are piecewise linear. In addition, in this paper we do not deal with crossing limit cycles, but we are interested in limit cycles of the non-linearly regularized vector field (which is $C^{\infty}$-smooth). If a limit cycle is located inside the regularization stripe, it shrinks to the switching manifold as $\varepsilon \rightarrow 0$. 

In what follows, the \emph{degree} of a continuous combination is the degree of $\widetilde{Z}_{\mu}$ with respect to the $\lambda$ variable. In Theorems \ref{thm-center} and \ref{thm-ii2}, the parameter $\mu$ is one dimensional. Moreover, we say that an invisible invisible fold-fold singularity is a \emph{codimension 1 non-smooth focus} if the first return map $\xi$ of $Z_{\mu}$ is given by $\xi(y) = y + ay^{2} + O(y^{3})$, with $a \neq 0$. The invisible invisible fold-fold singularity is a \emph{non-smooth center} if the first return map $\xi$ of $Z_{\mu}$ is given by $\xi(y) = y$. The definition of such singularities is also recalled in detail in section \ref{sec-candidates}.

\begin{mtheorem}\label{thm-center}
There exists a continuous combination $\widetilde{Z}_\mu$ of degree $2$ associated with a linear PSVF $Z_\mu$ having a non-smooth center such that the following is true: for any integer $k > 0$, there exist a monotonic transition function $\varphi_k$ and a continuous function $\mu_k:[0,\varepsilon_k]\to \mathbb{R}$, with $\varepsilon_k>0$, such that the non-linear regularization $\widetilde{Z}^{\varphi_k}_{\varepsilon,\mu_k(\varepsilon)}$ has at least $k+1$ hyperbolic limit cycles, for each $\varepsilon\in (0,\varepsilon_k]$.
\end{mtheorem}

The linear PSVF and continuous combination from Theorem \ref{thm-center} are given in \eqref{eq-II2-ns-center} and \eqref{eq-II2-center-cc}, respectively. We have a similar result for the singularity II$_{2}$.

\begin{mtheorem}\label{thm-ii2}
There exists a continuous combination $\widetilde{Z}_\mu$ of degree at least $4$ associated with a linear PSVF $Z_\mu$ having a codimension $1$ non-smooth focus such that the following is true: for any integer $k > 0$, there exist a monotonic transition function $\varphi_k$ and a continuous function $\mu_k:[0,\varepsilon_k]\to \mathbb{R}$, with $\varepsilon_k>0$, such that the non-linear regularization $\widetilde{Z}^{\varphi_k}_{\varepsilon,\mu_k(\varepsilon)}$ has at least $k+1$ hyperbolic limit cycles, for each $\varepsilon\in (0,\varepsilon_k]$.    
\end{mtheorem}

Concerning Theorems \ref{thm-center} and \ref{thm-ii2}, one should point the following remarks. Firstly, the $k+1$ limit cycles from Theorem \ref{thm-center} are either produced by red canard cycles in Figure \ref{fig-intro-crossing} (thus, each of them shrinks to the switching manifold as $\varepsilon \to 0$) or by blue canard cycles in Figure \ref{fig-intro-crossing}. We refer to the proof of Theorem \ref{thm-center} in Section \ref{application-center}. Secondly, the $k+1$ limit cycles from Theorem \ref{thm-ii2} shrink to the switching manifold as $\varepsilon \to 0$. For more details see the proof of Theorem \ref{thm-ii2} given in Section \ref{application-II2}.


Our strategy is to define a suitable non-linear regularization such that it presents a generic Hopf turning point. Using tools from geometric singular perturbation theory, in Theorem \ref{theo-main} in Section \ref{sec-terminal-case} we give a simple criterion for detecting such limit cycles, in terms of zeros of the slow divergence integrals. Finally, in Sections \ref{application-center} and \ref{application-II2} we show how to construct zeros for the slow divergence integral.

We point out that Theorem \ref{theo-main} (see also Figure \ref{fig-limit-cycles}(a)) is the main result of this paper which works under very general conditions. It is more convenient to state it later, after we introduce the notion of generic (Hopf) turning point, slow divergence integral, among other tools. We also prove Theorem \ref{theo-dodging} in Section \ref{section-dodging-case} for canard cycles of dodging type, see Figure \ref{fig-limit-cycles}(b). Such canard cycles, in contrast to Theorem \ref{theo-main}, can produce at most $2$ limit cycles. Theorem \ref{theo-dodging} also contains a simple criterion in terms of the slow divergence integral for existence of saddle node bifurcation of limit cycles.






\section{Non-linear regularizations with a generic turning point}\label{section-model-def}

In this section, we define a continuous combination whose regularization has a generic turning point positioned at the origin. Such singularity is the subject of subsection \ref{subsec-hopf}, where we also justify our approach by non-linear regularizations.

\subsection{Slow-fast Hopf points}\label{subsec-hopf}

We say that the 2-dimensional (smooth) slow-fast system
\begin{equation}\label{eq-def-slowfast-1}
\left\{
\begin{array}{rcl}
\dot{x} & = & f_{\mu}(x,y,\varepsilon), \\
\dot{y} & = & \varepsilon g_{\mu}(x,y,\varepsilon),
\end{array}
\right.
\end{equation}
has a slow-fast Hopf point positioned at the origin if, for a fixed parameter $\mu = \mu_{0}$, it satisfies \cite[Definition 2.4]{DMDR}
\begin{equation}\label{eq-def-hopf-turning-point}
\begin{split}
    f_{\mu_{0}}(0,0,0) = g_{\mu_{0}}(0,0,0) = \frac{\partial f_{\mu_{0}}}{\partial x}(0,0,0) = 0, \\ 
    \frac{\partial^{2} f_{\mu_{0}}}{\partial x^{2}}(0,0,0) \neq 0, \ \quad \ \left(\frac{\partial g_{\mu_{0}}}{\partial x}(0,0,0)\right)\left( \frac{\partial f_{\mu_{0}}}{\partial y}(0,0,0)\right) < 0. \
\end{split}
\end{equation}

Following \cite[Section 6.1]{DMDR}, a normal form for smooth equivalence for a slow-fast Hopf point is given by
\begin{equation}\label{sfHopfNormal}
\left\{
\begin{array}{rcl}
\dot{x} & = & y - x^{2} + x^{3}h_{1}(x,\varepsilon,\mu), \\
\dot{y} & = & \varepsilon\Big{(}a(\mu) - x + x^{2}h_{2}(x,\varepsilon,\mu) + yh_{3}(x,y,\varepsilon,\mu)\Big{),}
\end{array}
\right.
\end{equation}
where functions $a,h_1,h_2,h_3$ are smooth and $a(\mu_0)=0$. It is straightforward to see that system \eqref{sfHopfNormal} satisfies conditions \eqref{eq-def-hopf-turning-point}. The slow-fast Hopf point in \eqref{sfHopfNormal} is called a generic turning point if the function $\mu\mapsto a(\mu)$ is
a submersion at $\mu=\mu_0$. In this case we can take $\alpha=a(\mu)$ as a new independent parameter. The parameter $\alpha$ is called a breaking parameter and plays an important role when we want to create limit cycles of \eqref{sfHopfNormal}. We refer to \cite[Section 6.3]{DMDR} and later sections for more details. In particular, if $h_{2} \equiv h_{3} \equiv 0$ in \eqref{sfHopfNormal}, one obtains a slow-fast classical Liénard equation as studied in \cite{DeMaesschalckHuzak,DPR}. 


It is well-known that, when considering a regularized vector field as \eqref{eq-def-regularization}, after rescaling of the form $x = \varepsilon\widetilde{x}$ and multiplication by $\varepsilon$ one obtains a slow-fast system of the form \eqref{eq-def-slowfast-1}. More specifically, given a PSVF $Z_{\mu} = (X_{\mu},Y_{\mu})$ and applying the previous transformation and multiplication in a \emph{linear} regularization, one obtains the system
(dropping the tilde to simplify the notation)
\begin{equation}\label{eq-slow-fast-pwsvf-planar}
\left\{
\begin{array}{rcl}
   \dot{x} &=& \displaystyle\frac{X_{1} + Y_{1}}{2} + \varphi(x)\left(\displaystyle\frac{X_{1} - Y_{1}}{2}\right), \\ 
   \dot{y} &= &\varepsilon\left(\displaystyle\frac{X_{2} + Y_{2}}{2} + \varphi(x)\left(\displaystyle\frac{X_{2} - Y_{2}}{2}\right)\right),
\end{array}
\right.
\end{equation}
in which $X = (X_{1},X_{2})$, $Y = (Y_{1},Y_{2})$ are applied in $(\varepsilon x,y)$ (we briefly omitted the parameter $\mu$ for the sake of simplicity). With this configuration, in which the regularization is linear, the dynamics on the half cylinder does not present slow-fast Hopf point, with $\varphi$ being monotonic or not.
Indeed, suppose by contradiction that system \eqref{eq-slow-fast-pwsvf-planar} has a slow-fast Hopf point at the origin. Then one would have
\begin{equation*}
\begin{array}{ccc}
(X_{1} + Y_{1}) + \varphi(0)(X_{1} - Y_{1})  =  (X_{2} + Y_{2}) + \varphi(0)(X_{2} - Y_{2}) & = & 0, \\
\varphi'(0)(X_{1} - Y_{1}) & = & 0, \\
\Big{(}\varphi'(0)(X_{2} - Y_{2})\Big{)}\Big{(}(X_{1} + Y_{1})_{y} + \varphi(0)(X_{1} - Y_{1})_{y}\Big{)} & < & 0,
\end{array}
\end{equation*}
in which all of these functions are evaluated at $(0,0)$ (see Equation \eqref{eq-def-hopf-turning-point}). This would imply $\varphi'(0) = 0$ and $\varphi'(0) \neq 0$ simultaneously, which is a contradiction. We remark that one must require $(X_{1} - Y_{1})(0,0) \neq 0$, otherwise for $(y,\varepsilon)=(0,0)$ all terms $a_{n}x^{n}$ in the Taylor expansion of the first equation of \eqref{eq-slow-fast-pwsvf-planar} would be zero. 

As we will see in next subsection, one can generate this kind of point with non-linear regularizations. This also justifies our approach using such a regularization process.

\subsection{Piecewise smooth model}\label{sec-pws-model}

Motivated by the model \eqref{sfHopfNormal} given in Section \ref{subsec-hopf}, we consider the continuous combination
\begin{equation}\label{eq-combination-hopf}
\widetilde{Z}_{\mu}(\lambda,x,y) = \left\{
  \begin{array}{rcl}
   \widetilde{Z}_{1,\mu}(\lambda,x,y)  & = & y - \lambda^{2} + A_{\mu}(\lambda,x), \\
   \widetilde{Z}_{2,\mu}(\lambda,x,y)  & = & \alpha -\lambda + B_{\mu}(\lambda,x,y).
  \end{array}
\right.
\end{equation}

In Equation \eqref{eq-combination-hopf}, we assume $\mu=(\alpha,\widetilde\mu)$ where the parameter $\alpha$ plays the role of a breaking parameter kept near zero and $\widetilde \mu\in \mathbb{R}^{l-1}$ denotes extra parameters. The functions $A_{\mu}$ and $B_{\mu}$ are given by 
\begin{equation*}
  \begin{array}{rcl}
A_{\mu}(\lambda,x) & = & \displaystyle\sum_{i = 0}^{3}\lambda^{3 - i}x^{i}A_{i,\mu}(\lambda,x), \\
B_{\mu}(\lambda,x,y) & = & \displaystyle\sum_{j = 0}^{2}\lambda^{2 - j}x^{j}B_{j,\mu}(\lambda,x) +  y B_{3,\mu}(\lambda, x,y),
  \end{array}
\end{equation*}
with $A_{i,\mu}$ and $B_{j,\mu}$ being smooth functions for $i,j = 0,\dots,3$. The functions $A_{\mu}$ and $B_{\mu}$ denote higher order terms on the variables $\lambda, x$ and $y$, and they can affect the dynamics of the PSVF $Z_{\mu}$ near $\Sigma = \{x = 0\}$. In this case, the PSVF is given by
\begin{equation}\label{eq-psvf-hopf}
Z_{\mu}(x,y) = \left\{
  \begin{array}{ll}
   X_{\mu}(x,y) =  \big{(}y - 1+ A_{\mu}(1,x), & \alpha - 1 + B_{\mu}(1,x,y)\big{)}, \\
   Y_{\mu}(x,y) =  \big{(}y - 1 + A_{\mu}(-1,x), & \alpha + 1 + B_{\mu}(-1,x,y)\big{)}.
  \end{array}
\right.
\end{equation}

We define a non-linear regularization of $Z_{\mu} = (X_{\mu},Y_{\mu})$  as
\begin{equation}\label{eq-def-regularization-2first}
\widetilde{Z}^{\varphi}_{\varepsilon,\mu}(x,y):= \widetilde{Z}_{\mu}\Big{(}\varphi\Big{(}\frac{x}{\varepsilon}\Big{)},x,y\Big{)},    
\end{equation} 
where $\varphi$ is a transition function (not necessarily monotonic). After rescaling $x = \varepsilon\widetilde{x}$ and multiplication by $\varepsilon$ we get (we drop the tilde in order to simplify the notation)
\begin{equation}\label{eq-blow-up-e1-pomoc}
\left\{
  \begin{array}{rclcl}
    \dot{x}  & = & y - \varphi^{2}(x) +  A_{\mu}(\varphi(x),\varepsilon x), \\
    \dot{y}  & = &  \varepsilon\left(\alpha - \varphi(x)  +  B_{\mu}(\varphi(x),\varepsilon x,y) \right).
  \end{array}
\right.
\end{equation}

In what follows, we will define some important assumptions for our proofs. We first state them, and their role will be explained in the sections below.

Given a transition function $\varphi$, define the functions
\begin{equation}\label{eq-def-f-g}
F_{\widetilde\mu}(x) := \varphi^{2}(x) - A_{(0,\widetilde\mu)}(\varphi(x),0), \quad G_{\widetilde\mu}(x) := B_{(0,\widetilde\mu)}\big{(}\varphi(x),0,F_{\widetilde\mu}(x)\big{)}- \varphi(x).
\end{equation}

For some $0 < M_1, M_2 < 1$, we require the following assumptions for all $x\in [-M_1,M_2]$ and $\mu = \mu_0 = (0,\widetilde\mu_0)$ being a parameter defined as in the beginning of Section \ref{section-model-def}.

\begin{itemize}
    \item[\textbf{(A0)}] The transition function satisfies $\varphi(0)=0$ and $\varphi'(0)>0$.
    \item[\textbf{(A1)}] $A_{\mu_0}(\pm1,0) < 1$.
    \item[\textbf{(A2)}] The function $F$ defined in \eqref{eq-def-f-g} satisfies $\frac{F_{\widetilde\mu_0}'(x)}{x}>0$, for all $x\in [-M_1,M_2]$.
    \item[\textbf{(A3)}] The function $G$ defined in \eqref{eq-def-f-g} satisfies $\frac{G_{\widetilde\mu_0}(x)}{x}<0$, for all $x\in [-M_1,M_2]$.
\end{itemize}

Assumption \textbf{(A0)} and the definition of $A_{\mu}$ and $B_{\mu}$ imply that the slow-fast system \eqref{eq-blow-up-e1-pomoc} has a generic turning point at the origin for $\mu=\mu_0=(0,\widetilde\mu_0)$. Notice that using a rescaling in $(x,y,t)$ the system \eqref{eq-blow-up-e1-pomoc} can be brought, near the origin, into the normal form \eqref{sfHopfNormal}.

Assumptions \textbf{(A0)}, \textbf{(A2)}, \textbf{(A3)} will be relevant in Section \ref{subsec-blow-up-e1} when we define the notion of slow divergence integral. We use \textbf{(A1)} in Section \ref{subsec-fold-points} when we study tangency points. 

\subsection{Tangency points}\label{subsec-fold-points}

Recall that the switching locus is the set $\Sigma = \{x = 0\}$. We are interested in tangency points of the PSVF $Z_{\mu}=(X_{\mu},Y_{\mu})$ given by \eqref{eq-psvf-hopf}. Tangency points of $X_{\mu}$ and $Y_{\mu}$ will be denoted by $T_{\mu}^{X}$ and $T_{\mu}^{Y}$, respectively. It will be clear for the reader that the functions $A_{\mu}$ and $B_{\mu}$ in \eqref{eq-psvf-hopf} determine the position and the (in)visibility of fold points, respectively.

Since $X_{\mu}F=X_{1,\mu}$, the vector field $X_{\mu}=(X_{1,\mu},X_{2,\mu})$ is tangent to $\Sigma$ at points of the form $T_{\mu}^{X} = (0,y_{\mu}^{X})$ with
\begin{equation}\label{eq-tangency-X}
y_{\mu}^{X}=1 - A_{\mu}(1,0), 
\end{equation}

The second order Lie-derivative at $T_{\mu}^{X}$ is given by
$$X^{2}_{\mu} F(T_{\mu}^{X}) = X_{2,\mu}(T_{\mu}^{X}) = \alpha - 1 + B_{\mu}(1,0,y^{X}_{\mu}),$$
where we used $X_{1,\mu}(T_{\mu}^{X})=0$ and $\frac{\partial}{\partial y}X_{1,\mu}(x,y)=1$ (see Equation \eqref{eq-psvf-hopf}). This implies that $T_{\mu}^{X}$ is either a fold point of $X_\mu$ or a singularity of $X_\mu$. 

Analogously, $Y_{\mu}F=Y_{1,\mu}$ and tangency points of $Y_{\mu}$ are of the form $T_{\mu}^{Y} = (0,y_{\mu}^{Y})$, with
\begin{equation}\label{eq-tangency-Y}
y_{\mu}^{Y}=1 - A_{\mu}(-1,0), \quad Y^{2}_{\mu} F(T_{\mu}^{Y})= \alpha + 1 + B_{\mu}(-1,0,y_{\mu}^{Y}),
\end{equation}
therefore the point $T_{\mu}^{Y}$ is either a fold point of $Y_\mu$ or a singularity of $Y_\mu$. 
\smallskip

It can be easily seen that $$\Sigma^{w}=\{y<\min\{y_{\mu}^{X},y_{\mu}^{Y}\}\}\cup \{y>\max\{y_{\mu}^{X},y_{\mu}^{Y}\}\}, \ \ \Sigma^{s}=\{\min\{y_{\mu}^{X},y_{\mu}^{Y}\}<y<\max\{y_{\mu}^{X},y_{\mu}^{Y}\}\},$$
where $\Sigma^{w}$ is the sewing region and $\Sigma^{s}$ is the sliding region (see Section \ref{section-continuous-comb}). 
Notice that $\Sigma^{s}=\emptyset$ when $y_{\mu}^{X}=y_{\mu}^{Y}$.

The assumption \textbf{(A1)} implies that, for a fixed $\mu = \mu_0 = (0,\tilde{\mu}_0)$, the numbers $y_{\mu_0}^{X}$ and $y_{\mu_0}^{Y}$ defined in \eqref{eq-tangency-X} and \eqref{eq-tangency-Y}, respectively, are positive, therefore $T_{\mu_0}^{X}$ and $T_{\mu_0}^{Y}$ lie above the generic turning point of \eqref{eq-blow-up-e1-pomoc}. This assumption will be important when we define canard cycles as in Figure \ref{fig-limit-cycles}. 






\subsection{Scaling the breaking parameter}\label{subsec-blow-up}

Recall the continuous combination $\widetilde{Z}_{\mu}$ and the PSVF $Z_{\mu} = (X_{\mu},Y_{\mu})$ given by \eqref{eq-combination-hopf} and \eqref{eq-psvf-hopf}, respectively. We introduce the scaling
$$\alpha=\varepsilon\widetilde\alpha,$$
in which $\widetilde\alpha\sim 0$ is called a \emph{regular breaking parameter}. For our purposes we consider a non-linear regularization of $Z_{\mu} = (X_{\mu},Y_{\mu})$ as the family $\widetilde{Z}^{\varphi}_{\varepsilon,\widetilde\alpha,\widetilde\mu}$ given by 
\begin{equation}\label{eq-def-regularization-2}
\widetilde{Z}^{\varphi}_{\varepsilon,\widetilde\alpha,\widetilde\mu}(x,y):= \widetilde{Z}_{(\varepsilon\widetilde\alpha,\widetilde\mu)}\Big{(}\varphi\Big{(}\frac{x}{\varepsilon^{2}}\Big{)},x,y\Big{)}.    
\end{equation} 

\begin{remark}
{\rm In this paper we focus on the regularization $\widetilde{Z}^{\varphi}_{\varepsilon,\mu}$ defined as in Equation \eqref{eq-def-regularization-2first}, nevertheless, we work with $\widetilde{Z}^{\varphi}_{\varepsilon,\widetilde\alpha,\widetilde\mu}$  with $\widetilde\alpha$ being a regular breaking parameter in Equation \eqref{eq-def-regularization-2}. The main reason why we work with $\widetilde{Z}^{\varphi}_{\varepsilon,\widetilde\alpha,\widetilde\mu}$ instead of $\widetilde{Z}^{\varphi}_{\varepsilon,\mu}$ is that we can then use results from \cite{DeMaesschalckDumortier1,Dumortier} (after the rescaling $x = \varepsilon^2\widetilde{x}$ and multiplication by $\varepsilon^2$). Due to this fact, we adopt the notation $(\varepsilon\widetilde\alpha,\widetilde\mu)$ instead of $\mu$.}

\end{remark}


\subsection{Cylindrical blow-up}\label{section-cyl-BU}

In order to study the dynamics of $\widetilde{Z}^{\varphi}_{\varepsilon,\widetilde\alpha,\widetilde\mu}$ near the switching locus $\Sigma=\{x=0\}$, we perform a \emph{cylindrical blow-up} of the form
\begin{equation*}\label{eq-blow-up}
  \begin{array}{rccc}
\Phi: & \mathcal{M} & \rightarrow & \mathbb{R}^{3} \\
& (\tilde{x},y,\tilde{\varepsilon},\rho) & \mapsto & (\rho^{2} \tilde{x}, y, \rho\tilde{\varepsilon}) = (x,y,\varepsilon);
\end{array}
\end{equation*}
with $\mathcal{M}$ being a \emph{manifold with corners} (see \cite{Joyce} for details), $(\tilde{x},\tilde{\varepsilon})\in\mathbb{S}^{1}$, $y\in\mathbb{R}$ and $\tilde{\varepsilon},\rho \geq 0$. This blow-up has a slightly different expression from the usual approach \cite{BuzziSilvaTeixeira} (it is quasi-homogeneous instead of homogeneous). The \emph{blown-up vector field} is defined as the pullback of $\widetilde{Z}^{\varphi}_{\varepsilon,\widetilde\alpha,\widetilde\mu}+0\frac{\partial}{\partial\varepsilon}$ multiplied by $\rho^{2}$: 
$${Z}^{\varphi}_{\widetilde\alpha,\widetilde\mu}:=\rho^{2}\Phi^*\left(\widetilde{Z}^{\varphi}_{\varepsilon,\widetilde\alpha,\widetilde\mu}+0\frac{\partial}{\partial\varepsilon}\right).$$

In Sections \ref{subsec-blow-up-e1}--\ref{subsec-blow-up-u-1}, we study the dynamics of ${Z}^{\varphi}_{\widetilde\alpha,\widetilde\mu}$ near the exceptional divisor $\mathcal{C} = \{\Phi^{-1}(\Sigma)\}$, which is a half-cylinder (see Figure \ref{fig-cylindrical-blowup}). This study is carried out using directional charts.



\begin{figure}[ht]\center{
\begin{overpic}[width=0.5\textwidth]{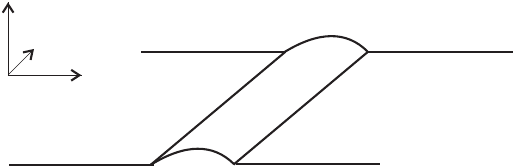}
\put(0,33){$\varepsilon$}
\put(7,23){$y$}
\put(17,17){$x$}
\end{overpic}}
\caption{\footnotesize{Blowing up the switching line $\Sigma$.}}
\label{fig-cylindrical-blowup}
\end{figure}

\subsubsection{Dynamics in the chart $\tilde{\varepsilon} = 1$}\label{subsec-blow-up-e1} 
In this scaling chart we have $x = \varepsilon^{2}x_{2}$ where $(x_2,y)$ is kept in a large compact set in $\mathbb R^2$ and $\varepsilon>0$ is small. The vector field $\widetilde{Z}^{\varphi}_{\varepsilon,\widetilde\alpha,\widetilde\mu}+0\frac{\partial}{\partial\varepsilon}$ yields (after multiplication by the positive factor $\varepsilon^2$) to
\begin{equation}\label{eq-blow-up-e1}
\left\{
  \begin{array}{rclcl}
    \dot{x}_{2}  & = & y - \varphi^{2}(x_{2}) & + & A_{(\varepsilon\widetilde\alpha,\widetilde\mu)}\left(\varphi(x_{2}),\varepsilon^{2}x_{2}\right), \\
    \dot{y}  & = & \varepsilon^{2}\Big{(}\varepsilon\widetilde\alpha - \varphi(x_{2}) & + & B_{(\varepsilon\widetilde\alpha,\widetilde\mu)}\left(\varphi(x_{2}),\varepsilon^{2}x_{2},y\right)\Big{)}, 
  \end{array}
\right.
\end{equation}
with $\dot\varepsilon=0$ ($\varepsilon$ is the singular perturbation parameter). For $\varepsilon=0$, equation \eqref{eq-blow-up-e1} turns to
\begin{equation}\label{eq-blow-up-e1-fast-sub}
\left\{
  \begin{array}{rclcl}
    \dot{x}_{2}  & = & y - \varphi^{2}(x_{2}) & + & A_{(0,\widetilde\mu)}(\varphi(x_{2}),0), \\
    \dot{y}  & = & 0. 
  \end{array}
\right.
\end{equation}

The curve of singularities of \eqref{eq-blow-up-e1-fast-sub} is given by $C_{0} = \big{\{}y =F_{\widetilde\mu}(x_2)\big{\}}$,
where $F_{\widetilde\mu}$ is defined in \eqref{eq-def-f-g}.
The curve $C_0$ contains two horizontal portions $\{y=y_{(0,\widetilde\mu)}^X, x_2\ge 1\}$ and $\{y=y_{(0,\widetilde\mu)}^Y, x_2\le -1\}$, where $y_{(0,\widetilde\mu)}^X$ is defined in \eqref{eq-tangency-X} and $y_{(0,\widetilde\mu)}^Y$ in \eqref{eq-tangency-Y}. See Figure \ref{fig-chart-e1}. 

Recall that in the interval $(-1,1)$ we adopt Assumption \textbf{(A2)}. Assumption \textbf{(A2)} is true locally near $x_2=0$, due to Assumption \textbf{(A0)}.  The Assumption \textbf{(A2)} implies that for each $\widetilde\mu\sim\widetilde\mu_0$ the curve $C_0$ has a parabola-like shape inside the segment $[-M_1,M_2]$ with a normally attracting branch ($x_2\in (0,M_2]$) and a normally repelling branch ($x_2\in [-M_1,0)$). It is clear that the transition function $\varphi$ has to be monotonic on $[-M_1,M_2]$ (not necessarily monotonic outside this segment).  

The slow dynamics \cite[Chapter 3]{DMDR} along the segment $[-M_1,M_2]$ is given by 
\begin{equation}\label{eq-blow-up-e1-slow}
{x_{2}'} = \displaystyle\frac{dx_{2}}{d\tau} = \displaystyle\frac{G_{\widetilde\mu}(x_2)}{F_{\widetilde\mu}'(x_2)},   
\end{equation}
in which $\tau=\varepsilon^2 t$ stands for the slow time (we denote by $t$ the fast time in \eqref{eq-blow-up-e1}). The Assumption \textbf{(A3)} defined in Section \ref{sec-pws-model} assures that the slow dynamics \eqref{eq-blow-up-e1-slow} is regular for all $x_2\in [-M_1,M_2]$ and points from the attracting branch to the repelling branch.

The following slow divergence integrals \cite[Chapter 5]{DMDR} (see also \cite{DeMaesschalckHuzak}) play an important role in this paper:
\begin{equation}\label{PWSSDI}
I_{\widetilde\mu}^+(x_2):= - \displaystyle\int_{x_2}^{0}\frac{F'_{\widetilde\mu}(s)^2}{G_{\widetilde\mu}(s)}ds<0, \ x_2\in (0,M_2], \quad  I_{\widetilde\mu}^-(x_2):= - \displaystyle\int_{x_2}^{0}\frac{F'_{\widetilde\mu}(s)^2}{G_{\widetilde\mu}(s)}ds<0, \ x_2\in [-M_1,0).
\end{equation}

These are integrals of the divergence of the vector field \eqref{eq-blow-up-e1-fast-sub}, computed along $C_{0}$ with respect to the slow time $\tau$. More precisely, the integral $I_{\widetilde\mu}^+(x_2)$ is computed along the attracting segment $[0,x_2]\subset [0,M_2]$ and $I_{\widetilde\mu}^-(x_2)$ along the repelling segment $[x_2,0]\subset [-M_1,0]$. See also Sections \ref{sec-terminal-case} and \ref{section-dodging-case}.

\begin{figure}[ht]\center{
\begin{overpic}[width=0.45\textwidth]{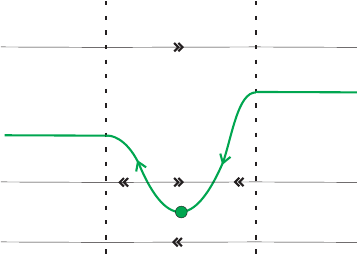}
\put(20,-4){$x_{2} = -1$}
\put(63,-4){$x_{2} = 1$}
\end{overpic}}
\caption{\footnotesize{Phase portrait of \eqref{eq-blow-up-e1-fast-sub} (the family chart $\tilde{\varepsilon} = 1$).}}
\label{fig-chart-e1}
\end{figure}

\subsubsection{Dynamics in the chart $\tilde{x} = 1$}\label{subsec-blow-up-u1} In this chart we have
$(x,\varepsilon) = (\rho_{3}^{2},\rho_{3}\varepsilon_{3})$. The vector field $\widetilde{Z}^{\varphi}_{\varepsilon,\widetilde\alpha,\widetilde\mu}+0\frac{\partial}{\partial\varepsilon}$ changes, after multiplication by $\rho_{3}^{2}$, into  
\begin{equation}\label{eq-blow-up-u1}
\left\{
  \begin{array}{rcl}
    \dot\rho_{3} & = & \rho_{3}H_{\widetilde\alpha,\widetilde\mu}(\rho_{3}, y, \varepsilon_{3}), \\
    \dot y & = & \rho_{3}^{2}\Big{(}\rho_{3}\varepsilon_{3}\widetilde\alpha - \varphi\big{(}\frac{1}{\varepsilon_{3}^{2}}\big{)} + B_{(\rho_{3}\varepsilon_{3}\widetilde\alpha,\widetilde\mu)}\Big{(}\varphi\big{(}\frac{1}{\varepsilon_{3}^{2}}\big{)},\rho_{3}^{2},y\Big{)}\Big{)}, \\
   \dot \varepsilon_{3} & = & -\varepsilon_{3}H_{\widetilde\alpha,\widetilde\mu}(\rho_{3}, y, \varepsilon_{3}),
  \end{array}
\right.       
\end{equation}
in which
$$H_{\widetilde\alpha,\widetilde\mu}(\rho_{3}, y, \varepsilon_{3}) =\frac{1}{2}\Big{(} y -\varphi^{2}\big{(}\frac{1}{\varepsilon_{3}^{2}}\big{)} + A_{(\rho_{3}\varepsilon_{3}\widetilde\alpha,\widetilde\mu)}\Big{(}\varphi\big{(}\frac{1}{\varepsilon_{3}^{2}}\big{)},\rho_{3}^{2}\Big{)}\Big{)}.$$

Observe that in \eqref{eq-blow-up-u1} we can set $\varphi\big{(}\frac{1}{\varepsilon_{3}^{2}}\big{)} = 1$ for $0\leq \varepsilon_{3} \leq 1$. 
The line $\{\rho_{3} = \varepsilon_{3} = 0\}$ consists of singularities of \eqref{eq-blow-up-u1}. The Jacobian matrix of \eqref{eq-blow-up-u1} evaluated in $(0,y,0)$ is given by
$$J_{3} = \left(
  \begin{array}{ccc}
    \frac{y-y_{(0,\widetilde\mu)}^X}{2} & 0 & 0\\
    0 & 0 & 0 \\
    0 & 0 & -\frac{y-y_{(0,\widetilde\mu)}^X}{2}  \\
  \end{array}
\right),$$
which means that all the singularities are semi-hyperbolic, except for $T_{(0,\widetilde\mu)}^{X}=(0,y_{(0,\widetilde\mu)}^X,0)$ where we deal with a degenerate singularity.

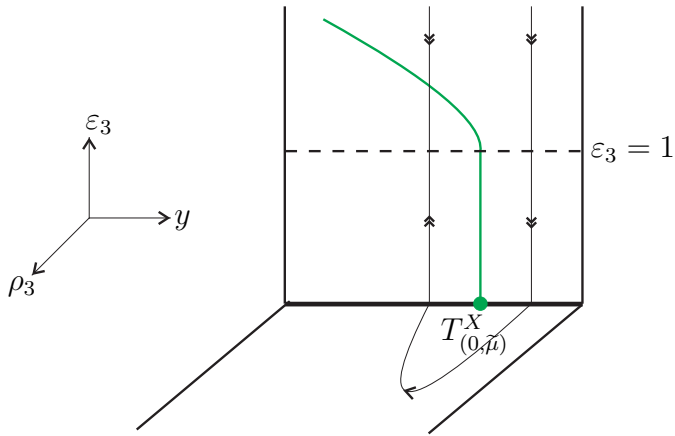
\begin{figure}[ht]\center{
\begin{overpic}[width=0.45\textwidth]{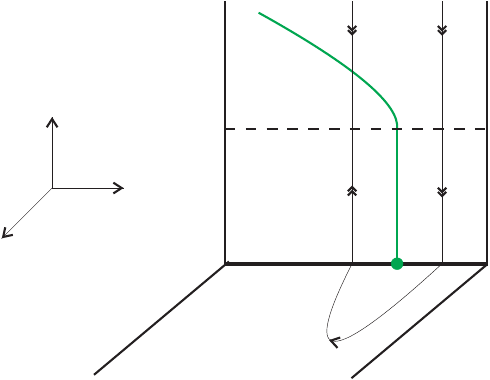}
\put(10,55){$\varepsilon_{3}$}
\put(-4,26){$\rho_{3}$}
\put(26,38){$y$}
\put(74,17){$T_{(0,\widetilde\mu)}^{X}$}
\put(101,50){$\varepsilon_{3} = 1$}
\end{overpic}}
\caption{\footnotesize{Phase portrait of \eqref{eq-blow-up-u1} (the chart $\tilde{x} = 1$). On the invariant plane $\{\rho_{3} = 0\}$ one obtains a smooth curve of singularities given by $\{H_{\widetilde\alpha,\widetilde\mu}(0, y, \varepsilon_{3}) = 0\}$. Such curve is a straight segment for $0 \leq \varepsilon_{3} \leq 1$ and it contains the point $T_{(0,\widetilde\mu)}^{X}$.}}
\label{fig-chart-u1}
\end{figure}

\subsubsection{Dynamics in the chart $\tilde{x} = -1$}\label{subsec-blow-up-u-1}

Using $(x,\varepsilon) = (-\rho_{1}^{2},\rho_{1}\varepsilon_{1})$ one obtains, after multiplication by $\rho_{1}^{2}$, the system
\begin{equation}\label{eq-blow-up-u-1}
\left\{
  \begin{array}{rcl}
    \rho'_{1} & = & -\rho_{1}G_{\widetilde\alpha,\widetilde\mu}(\rho_{1}, y, \varepsilon_{1}), \\
    y' & = & \rho_{1}^{2}\Big{(}\rho_{1}\varepsilon_{1}\widetilde{\alpha} - \varphi\big{(}-\frac{1}{\varepsilon_{1}^{2}}\big{)} + B_{(\rho_{1}\varepsilon_{1}\widetilde\alpha,\widetilde\mu)}\Big{(}\varphi\big{(}-\frac{1}{\varepsilon_{1}^{2}}\big{)},-\rho_{1}^{2},y\Big{)}\Big{)}, \\
    \varepsilon'_{1} & = & \varepsilon_{1}G_{\widetilde\alpha,\widetilde\mu}(\rho_{1}, y, \varepsilon_{1}),
  \end{array}
\right.  
\end{equation}
in which
$$G_{\widetilde\alpha,\widetilde\mu}(\rho_{1}, y, \varepsilon_{1}) =\frac{1}{2}\Big{(} y -\varphi^{2}\big{(}-\frac{1}{\varepsilon_{1}^{2}}\big{)} + A_{(\rho_{1}\varepsilon_{1}\widetilde\alpha,\widetilde\mu)}\Big{(}\varphi\big{(}-\frac{1}{\varepsilon_{1}^{2}}\big{)},-\rho_{1}^{2}\Big{)}\Big{)}.$$

One can set $\varphi\big{(}-\frac{1}{\varepsilon_{1}^{2}}\big{)} = - 1$ for $0 \leq \varepsilon_{1} \leq 1$ and the phase-portrait of Equation \eqref{eq-blow-up-u-1} can be sketched analogously as it was described in Subsection \ref{subsec-blow-up-u1}.

\section{Candidates for limit cycles}\label{sec-candidates}

Recall that $\alpha = \varepsilon\widetilde\alpha$, and we are considering the parameter $\mu$ as $\mu = (\alpha,\widetilde\mu) = (\varepsilon\widetilde\alpha,\widetilde\mu)$. See also the beginning of Subsection \ref{subsec-blow-up}.

Define the \emph{half return map of $X$} as
\begin{equation*}\label{eq-def-half-return}
  \begin{array}{rccl}
\xi_{X}: & \{y > y_{(0,\widetilde\mu)}^{X}\} \subset\Sigma & \rightarrow & \{y < y_{(0,\widetilde\mu)}^{X}\}\subset\Sigma\\
& p = (0,y) & \mapsto & \xi_{X}(y) = \pi_{2}\Big{(}\phi_{X}\big{(}t(p), p\big{)}\Big{)}, 
  \end{array}
\end{equation*}
in which $\pi_{2}$ is the projection in the second coordinate, $\phi_{X}$ is the flow of $X_{(0,\widetilde\mu)}$, and $t(p) > 0$ denotes the smallest (and finite) time in which the flow by $p$ intersects $\Sigma$. The half return map $\xi_{X}$ defined as above is smooth. The \emph{half return map of $Y_{(0,\widetilde\mu)}$} is denoted by $\xi_{Y}$ and it is defined in the same fashion, but we consider the flow of $Y_{(0,\widetilde\mu)}$ in backward time instead. See Figure \ref{fig-half-return}.

\begin{remark}
{\rm Of course, the half return maps will depend on the parameter $\widetilde\mu$. However, for simplicity sake, we will denote them by $\xi_{X,Y}$.}   
\end{remark}

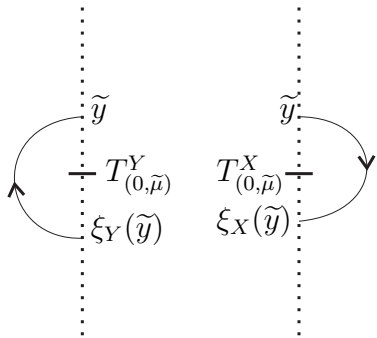
\begin{figure}[ht]\center{
\begin{overpic}[width=0.3\textwidth]{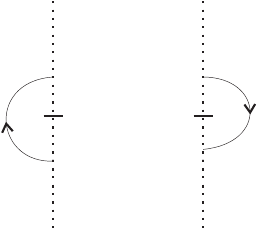}
\put(23,59){$\widetilde{y}$}
\put(74,59){$\widetilde{y}$}
\put(27,42){$T_{(0,\widetilde\mu)}^{Y}$}
\put(57,42){$T_{(0,\widetilde\mu)}^{X}$}
\put(23,27){$\xi_{Y}(\widetilde{y})$}
\put(57,29){$\xi_{X}(\widetilde{y})$}
\end{overpic}}
\caption{\footnotesize{Half-return maps $\xi_{X}$ and $\xi_{Y}$. Such maps are defined by considering the flow of $X_{(0,\widetilde\mu)}$ and $Y_{(0,\widetilde\mu)}$ in forward and backward time, respectively. }}
\label{fig-half-return}
\end{figure}

The \emph{first return map of $Z$} is given by $\xi = \xi_{Y}^{-1}\circ\xi_{X}(y)$. If $\xi(y) = y + ay^{2} + O(y^{3})$ with $a \neq 0$, then the invisible-invisible fold-fold has codimension 1, and it was denoted by II$_{2}$ in \cite{KGR} (see also \cite{GST}). Such singularity present a focus-like behavior. Moreover, if $a < 0$ then the singularity is an attracting non-smooth focus and it is a repelling one otherwise. If $\xi(y) = Id$, then the invisble-invisible fold-fold is a non-smooth center \cite{BuzziCarvalhoTeixeira}. We assume that the half return maps $\xi_{X,Y}(y)$ are well defined in a segment $\sigma_{0}$ as described in Section \ref{sec-sections-transitions}.

Consider the dynamics on the top of the half cylinder as described in Subsection \ref{subsec-blow-up-e1}. Since the origin is a generic Hopf turning point, it is clear that one could also expect creation of limit cycles from orbits inside the region $|x_{2}| < 1$ (recall Figure \ref{fig-chart-e1}). However, we also want to study ``big'' limit cycles in the sense that they are not strictly contained inside the regularization stripe. In what follows, we precisely define the shape of the limit periodic sets (called canard cycles), at level $\varepsilon = 0$ and $\widetilde\mu = \widetilde\mu_{0}$, that we are interested in. Two cases will be considered: \emph{terminal} and \emph{dodging cases}. This terminology is inspired in \cite[Lemma 5.3]{DMDR}.

Firstly we describe the terminal case (see Figure \ref{fig-limit-cycles}(a)). A canard cycle is the concatenation of orbits as follows. Starting on the top of the cylinder, we follow the orientation of a fast orbit in a level $\widetilde{y} > y_{(0,\widetilde\mu_{0})}^{X}, y_{(0,\widetilde\mu_{0})}^{Y}$ until one reaches the corner of the half cylinder. Then one continues through orbits of $X_{(0,\widetilde\mu_{0})}$ and, after reach the corner of the cylinder once again, we follow fast orbits until we hit the attracting branch of the critical manifold $C_0$ at the point with the $x_2$ coordinate contained in the interval $(0,M_2]$. One continues by the slow flow along $C_{0}$, crosses the turning point and continues along the reppeling branch of the critical manifold. Then one leaves the repelling branch at a point with $x_2$ coordinate contained in $[-M_1,0)$, follows the corresponding fast orbit until the corner of the half cylinder, continues by the flow of $Y_{(0,\widetilde\mu_{0})}$ and finally comes back to the same level $\widetilde{y}$ on the top of the cylinder. We assume that such a canard cycle exists and we denote it by $\Upsilon^{\widetilde{y}}$. See Figure \ref{fig-limit-cycles}(a). Recall the constants $M_{1,2}$ from Section \ref{sec-pws-model}.

Now, we describe the dodging case (see Figure \ref{fig-limit-cycles}(b)). Without loss of generality, assume $y_{(0,\widetilde\mu_{0})}^{Y} <y_{(0,\widetilde\mu_{0})}^{X}$. In this case the canard cycle is the concatenation as follows. Starting on the top of the cylinder, we follow the orientation of a fast orbit in a level $y_{(0,\widetilde\mu_{0})}^{Y} < \widetilde{y}<y_{(0,\widetilde\mu_{0})}^{X}$ until one reaches the attracting branch of $C_0$ at a point with $x_2$ coordinate contained in $(0,M_2]$. Now we proceed as in the previous case: one continues by the slow flow, crosses the turning point and continues along the reppeling branch of $C_{0}$. Then one leaves $C_0$ at a point with $x_2$ in $[-M_1,0)$, follows the corresponding fast orbit until the corner of the half cylinder, continues by the flow of $Y_{(0,\widetilde\mu_{0})}$ and finally comes back to the same level $\widetilde{y}$ on the top of the cylinder. We assume that such a canard cycle exists and denote it by $\Gamma^{\widetilde{y}}$. See Figure \ref{fig-limit-cycles}(b). We use a similar definition when $y_{(0,\widetilde\mu_{0})}^{X} <y_{(0,\widetilde\mu_{0})}^{Y}$.

There are canard cycles in the singular limit $\varepsilon = 0$ that can also generate limit cycles for $\varepsilon > 0$ and they will not be covered in this paper, since such cases are degenerate in the sense that they require further blow-up analysis in order to define transition maps. Such degenerate cases are characterized by having fast orbit segments in the level $y = y_{(0,\widetilde\mu_{0})}^{X}$ or $y = y_{(0,\widetilde\mu_{0})}^{Y}$ on the top of the cylinder. See Figure \ref{fig-degenerate-cases}.

\begin{figure}[ht]\center{
\begin{overpic}[width=0.7\textwidth]{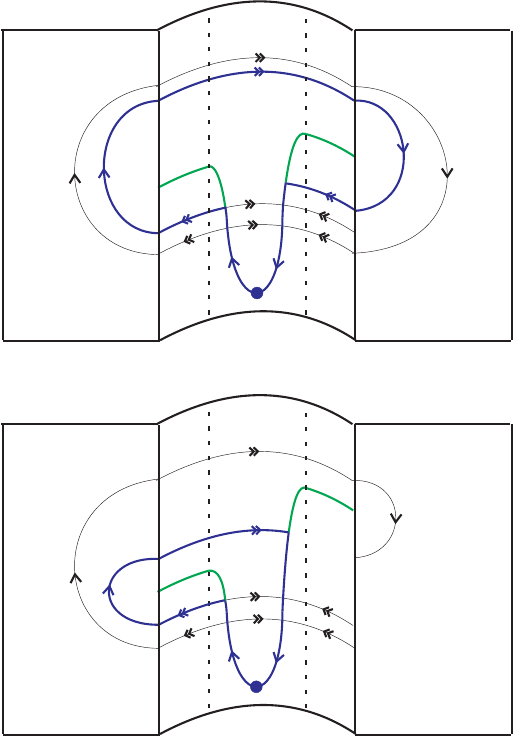}

\put(35,29){$\Gamma^{\tilde y}$}
\put(35,87){$\Upsilon^{\tilde y}$}

\put(-5,75){$\textbf{(a)}$}
\put(-5,23){$\textbf{(b)}$}

\put(22,47){$x_{2} = -1$}
\put(38,47){$x_{2} = 1$}

\put(21,100){$x_{2} = -1$}
\put(39,100){$x_{2} = 1$}

\put(48,78){$T_{(0,\widetilde\mu_{0})}^{X}$}
\put(15,76){$T_{(0,\widetilde\mu_{0})}^{Y}$}

\put(48,29){$T_{(0,\widetilde\mu_{0})}^{X}$}
\put(15,20){$T_{(0,\widetilde\mu_{0})}^{Y}$}
\end{overpic}}
\caption{\footnotesize{Canard cycles $\Upsilon^{\widetilde{y}}$ and $\Gamma^{\widetilde{y}}$ (highlighted in blue) that are candidates for limit cycles of $\widetilde{Z}^{\varphi}_{\varepsilon,\widetilde\alpha,\widetilde\mu}$. Figures (a) and (b) (on the top and on the bottom, respectively) show the terminal and dodging cases, respectively.}}
\label{fig-limit-cycles}
\end{figure}

\begin{figure}[ht]\center{
\begin{overpic}[width=0.7\textwidth]{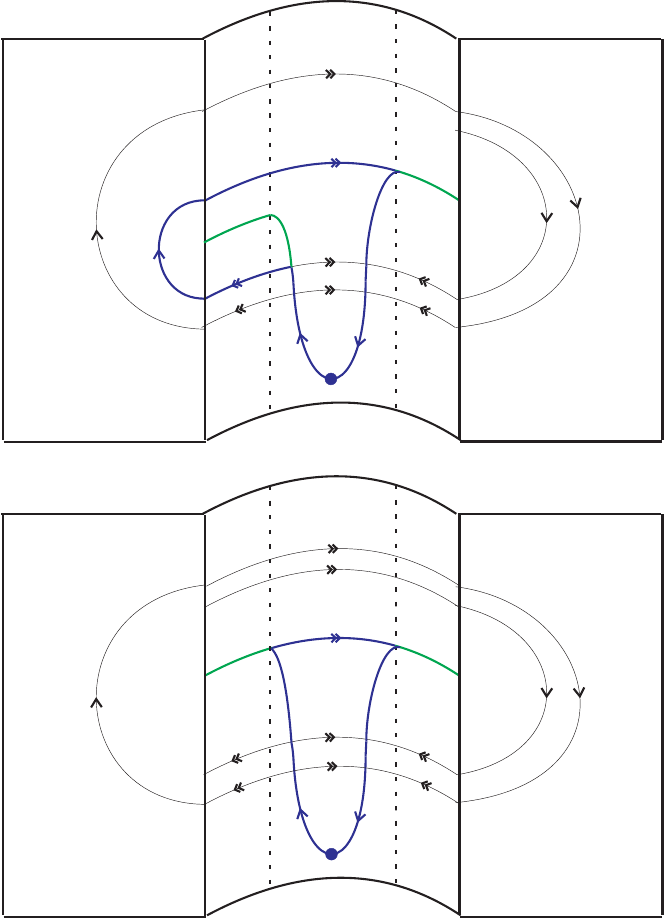}
\put(21,48){$x_{2} = -1$}
\put(41,48){$x_{2} = 1$}

\put(21,100){$x_{2} = -1$}
\put(41,100){$x_{2} = 1$}

\put(51,77){$T_{(0,\widetilde\mu_{0})}^{X}$}
\put(18,73){$T_{(0,\widetilde\mu_{0})}^{Y}$}

\put(51,26){$T_{(0,\widetilde\mu_{0})}^{X}$}
\put(15,26){$T_{(0,\widetilde\mu_{0})}^{Y}$}
\end{overpic}}
\caption{\footnotesize{Degenerate canard cycles (highlighted in blue) that can produce limit cycles of $\widetilde{Z}^{\varphi}_{\varepsilon,\widetilde\alpha,\widetilde\mu}$.}}
\label{fig-degenerate-cases}
\end{figure}






\section{Analysis of the terminal case}\label{sec-terminal-case}

In this section we prove that the slow divergence integral indeed can be used in order to detect limit cycles of the regularized vector field $\widetilde{Z}^{\varphi}_{\varepsilon,\widetilde\alpha,\widetilde\mu}$ (as in \eqref{eq-def-regularization-2}), for $\varepsilon$ small and positive. We consider the terminal case defined in Section \ref{sec-candidates} and assume that the assumptions \textbf{(A0)}-\textbf{(A3)} are satisfied.

The transversal sections will be denoted by $\sigma_{i}^{\pm}$ for $i=0,\dots,5$, and we simply denote $\sigma_{0}^{\pm} = \sigma_{0}$ and $\sigma_{5}^{\pm} = \sigma_{5}$. The transition maps will be denoted by $\Pi_{i}^{\pm}: \sigma_{i-1}^{\pm}\rightarrow \sigma_{i}^{\pm}$ for $i=1,\dots,5$. We will compute the transversal sections and transition maps in the right-hand side of Figure \ref{fig-transversal-sections}(a). The study of the left hand side is completely analogous.

\begin{figure}[ht]\center{
\begin{overpic}[width=0.6\textwidth]{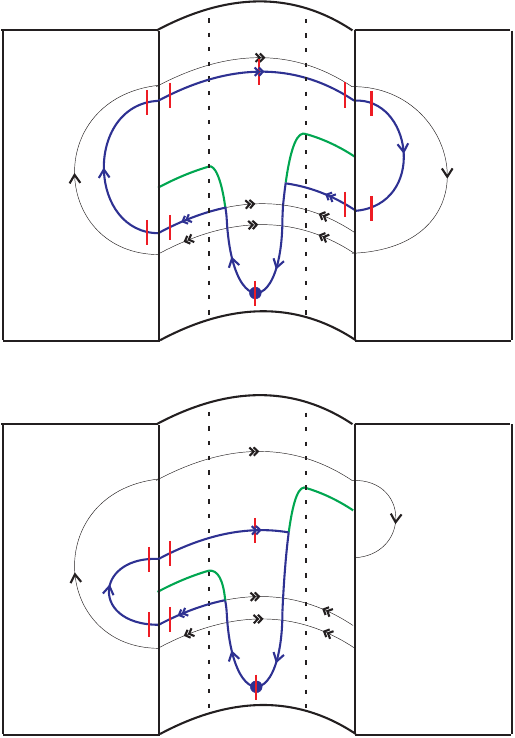}
\put(-5,75){$\textbf{(a)}$}
\put(-5,23){$\textbf{(b)}$}

\put(18,82){$\sigma_{2}^{-}$}
\put(23,84){$\sigma_{1}^{-}$}
\put(34,87){$\sigma_{0}$}
\put(43,84){$\sigma_{1}^{+}$}
\put(49,82){$\sigma_{2}^{+}$}

\put(18,71){$\sigma_{3}^{-}$}
\put(22,72){$\sigma_{4}^{-}$}
\put(34,63){$\sigma_{5}$}
\put(44,75){$\sigma_{4}^{+}$}
\put(49,75){$\sigma_{3}^{+}$}

\put(18,27){$\sigma_{2}^{-}$}
\put(22,28){$\sigma_{1}^{-}$}
\put(34,25){$\sigma_{0}$}

\put(17,17){$\sigma_{3}^{-}$}
\put(22,18){$\sigma_{4}^{-}$}
\put(34,9){$\sigma_{5}$}

\end{overpic}}
\caption{\footnotesize{Transversal sections highlighted in red. Figures (a) and (b) represent terminal and dodging cases, respectively.}}
\label{fig-transversal-sections}
\end{figure}

\subsection{Sections and transitions on the family chart $\tilde{\varepsilon} = 1$}\label{sec-sections-transitions}

We assume that there is $\widetilde{y}\in \{y > y_{(0,\widetilde\mu_{0})}^{X}, y_{(0,\widetilde\mu_{0})}^{Y}\}$ such that $\xi_{X}(\widetilde{y}) \in \{0 < y < y_{(0,\widetilde\mu_{0})}^{X}\}$ and $\xi_{Y}(\widetilde{y}) \in \{0 < y < y_{(0,\widetilde\mu_{0})}^{Y}\}$. The canard cycle through $\widetilde{y}$ sketched in Figure \ref{fig-limit-cycles}(a) and described in Section \ref{sec-candidates} must be well defined. We fix an open interval $S_0$ containing $\widetilde{y}$ and assume that this property is true for every $y\in S_0$.


In this family chart, we define the sections (as sketched in Figure \ref{fig-transversal-sections}(a))
\begin{equation*}
    \begin{array}{rcl}
\sigma_{0} & = &\{x_{2} = 0, \quad 0 \leq \varepsilon < \varepsilon_{0}, \quad y\in S_0 \}, \\
\sigma_{1}^{+} & = & \{x_{2} = \bar{x}_{2}, \quad 0 \leq \varepsilon < \varepsilon_{0}, \quad y\in S_{1} \}, \\   
\sigma_{4}^{+} & = & \{x_{2} = \bar{x}_{2}, \quad 0 \leq \varepsilon < \varepsilon_{0}, \quad y \in S_{4} \}, 
    \end{array}
\end{equation*}
in which $\bar{x}_{2} > 1$ is large and fixed and $\varepsilon_{0}$ is small. Moreover, $S_{1}$ and $S_{4}$ are open intervals containing $\widetilde{y}$ and $\xi_{X}(\widetilde{y})$, respectively. The sections $\sigma_{0}$, $\sigma_{1}^{+}$ and $\sigma_{4}^{+}$ are parametrized by the variable $y$. Following \cite{DeMaesschalckDumortier1}, the section $\sigma_{5}$  is defined near the  generic turning point $(x_2,y)=(0,0)$, but in a new phase space $(\hat x_2,\hat y)$ after applying the rescaling $(x_2,y)=(\varepsilon^2\hat x_2,\varepsilon\hat y)$ to \eqref{eq-blow-up-e1} and dividing the new system by $\varepsilon>0$. Then we take $\sigma_{5} = \{\hat x_{2} = 0, \ 0 \leq \varepsilon < \varepsilon_{0}, \ \hat y \in \mathbb R \}$.  

We have the following Lemma.

\begin{lemma}\label{lemma-p1+}
The transition map $\Pi_{1}^{+}:\sigma_{0}\rightarrow \sigma_{1}^{+}$ is well defined and it is $C^{\infty}$. It depends on $(y,\varepsilon,\widetilde\alpha,\widetilde\mu)$.   
\end{lemma}
\begin{proof}
The transition map $\Pi_{1}^{+}$ is defined by following trajectories of the slow-fast system \eqref{eq-blow-up-e1} between $\sigma_0$ and $\sigma_{1}^{+}$. Since \eqref{eq-blow-up-e1} is regular in this region, the statement is true.    
\end{proof}

Since the origin of \eqref{eq-blow-up-e1} is a generic turning point and the slow dynamics is regular in $[-M_{1},M_{2}]$, then we are in the framework of \cite{DeMaesschalckDumortier1}. Therefore, by \cite[Theorem 4]{DeMaesschalckDumortier1} one obtains the transition map $\Pi_{5}^{+}:\sigma_{4}^{+}\rightarrow\sigma_{5}$ and its properties. This transition map is defined by following the orbits of \eqref{eq-blow-up-e1} between sections $\sigma_4^+$ and $\sigma_5$. See Lemma \ref{lemma-p5+} below.

In what follows, we define
\begin{equation}\label{eq-sdi-i5}
I_{5}^{\pm}(y,\widetilde\mu) := I_{\widetilde\mu}^{\pm}(x_{2}^{\pm}(y)),
\end{equation}
where $I_{\widetilde\mu}^{+}$ (resp. $I_{\widetilde\mu}^{-}$) is the slow divergence integral \eqref{PWSSDI} and $x_{2}^{+}(y)>0$ (resp. $x_{2}^{-}(y)<0$) is the $x_{2}$-coordinate of the $\omega$-limit point (resp. $\alpha$-limit point) of the fast orbit of \eqref{eq-blow-up-e1-fast-sub} through $y\in\sigma_{4}^{+}$ (resp. $y\in\sigma_{4}^{-}$). Clearly, we have $F_{\widetilde\mu}(x_{2}^{\pm}(y))=y$ with $F_{\widetilde\mu}$ defined in \eqref{eq-def-f-g}.

\begin{lemma}\label{lemma-p5+}
The transition map $\Pi_{5}^{+}:\sigma_{4}^{+}\rightarrow\sigma_{5}$ is given by
\begin{equation*}
 \Pi_{5}^{+}(y,\varepsilon,\widetilde\alpha,\widetilde\mu) = f_{5}^{+}(\varepsilon,\widetilde\alpha,\widetilde\mu) - \operatorname{exp}\Bigg{(}\frac{I_{5}^{+}(y,\widetilde\mu) + \Psi_{1}^{+}(y,\varepsilon,\widetilde\alpha,\widetilde\mu) + \Psi_{2}^{+}(\varepsilon,\widetilde\alpha,\widetilde\mu)\varepsilon^{2}\ln\varepsilon}{\varepsilon^{2}}\Bigg{)},   
\end{equation*}
in which $I_{5}^{+}(y,\widetilde{\mu}) < 0$ is the slow divergence integral \eqref{eq-sdi-i5}. The functions $\Psi_{1,2}^{+}$ and $f_{5}^{+}$ are smooth, and $\Psi_{1}^{+}$ is $O(\varepsilon)$.
\end{lemma}

\begin{remark}
{\rm The minus sign in the front of the exponential in Lemma \ref{lemma-p5+} is due to the chosen orientation of parametrization of $\sigma_{5}$.}    
\end{remark}

To study the transitions $\Pi_{2,3,4}^{+}$ we will work on chart $\tilde{x} = 1$.

\subsection{Sections and transitions on the family chart $\tilde{x} = 1$} Firstly, we rewrite the transversal sections $\sigma_{1,4}^{+}$ in this system of coordinates (recall the directional charts described in subsection \ref{subsec-blow-up}):
\begin{equation*}
    \begin{array}{rcl}
\sigma_{1}^{+} & = & \{0 \leq \rho_{3} < \varepsilon_{0}\sqrt{\bar{x}_{2}}, \quad \varepsilon_{3} = \frac{1}{\sqrt{\bar{x}_{2}}}, \quad y \in S_{1}\}, \\  
\sigma_{4}^{+} & = & \{0 \leq \rho_{3} < \varepsilon_{0}\sqrt{\bar{x}_{2}}, \quad \varepsilon_{3} = \frac{1}{\sqrt{\bar{x}_{2}}}, \quad y \in S_{4}\},
    \end{array}
\end{equation*}
in which $\bar{x}_{2} > 1$. We further define the following transversal sections, with $\tilde{\rho}_{3}$ being a small and positive constant:
\begin{equation*}
    \begin{array}{rcl}
\sigma_{2}^{+} & = & \{\rho_{3} = \tilde{\rho}_{3}, \quad 0 \leq \varepsilon_{3} < \tilde{\varepsilon}_{0}, \quad y \in S_2\}, \\  
\sigma_{3}^{+} & = & \{\rho_{3} = \tilde{\rho}_{3}, \quad 0 \leq \varepsilon_{3} < \tilde{\varepsilon}_{0}, \quad y \in S_3\},
    \end{array}
\end{equation*}
where $S_{2}$ and $S_{3}$ are open intervals containing $\widetilde{y}$ and $\xi_{X}(\widetilde{y})$, respectively. The following Lemmas state properties of the transition maps $\Pi_{2,3,4}^{+}$. The smoothness of the flow of $\widetilde{Z}^{\varphi}_{\varepsilon,\widetilde\alpha,\widetilde\mu}$ assures Lemma \ref{lemma-p3+}.

\begin{lemma}\label{lemma-p3+}
The transition map $\Pi_{3}^{+}:\sigma_{2}^{+}\rightarrow\sigma_{3}^{+}$ is well defined and it is $C^{\infty}$. Moreover, it depends on $(y,\frac{\varepsilon}{\tilde{\rho}_{3}},\widetilde\alpha,\widetilde\mu)$.    
\end{lemma}

\begin{proof}
The transition map $\Pi_{3}^{+}$ is defined by following trajectories of \eqref{eq-blow-up-u1} between the sections $\sigma_{2}^{+}$ and $\sigma_{3}^{+}$. Since we assumed that this passage is regular, Lemma \ref{lemma-p3+} follows directly.
\end{proof}

The transition maps $\Pi_{2,4}^{+}$ were studied in \cite[Subsection 4.2]{GucwaSzmolyan}, \cite[Appendix B]{HuzakKristiansen} and \cite[Subsections 3.3 and 3.5]{WangZhang}. We state the properties that we need in the next Lemma.

\begin{lemma}\label{lemma-p24+}
The transition maps $\Pi_{2}^{+}:\sigma_{1}^{+}\rightarrow\sigma_{2}^{+}$ and $\Pi_{4}^{+}:\sigma_{3}^{+}\rightarrow\sigma_{4}^{+}$ are smooth and they are given by
\begin{equation*}
\left.
  \begin{array}{lcc}
    \Pi_{2}^{+}\left(y,\varepsilon\sqrt{\bar{x}_{2}},\widetilde\alpha,\widetilde\mu\right)  & = & g_{2}^{+}(y,\widetilde\alpha,\widetilde\mu) + O(\varepsilon\ln\varepsilon^{-1}), \\
   \Pi_{4}^{+}\left(y,\frac{\varepsilon}{\tilde{\rho}_{3}},\widetilde\alpha,\widetilde\mu\right)  & = & g_{4}^{+}(y,\widetilde\alpha,\widetilde\mu) + O(\varepsilon\ln\varepsilon^{-1}).  
  \end{array}
\right.
\end{equation*}
Moreover, these expressions can be differentiated with respect to $y$ without changing the order of the remainder.
\end{lemma}

\begin{remark}
{\rm In fact, Lemma \ref{lemma-p24+} gives the $y$-component of the transition maps $\Pi_{2,4}^{+}$. The $\varepsilon_{3}$ component of $\Pi_{2}^{+}$ is $\frac{\varepsilon}{\widetilde{\rho}_{3}}$ and the $\rho_{3}$ component of $\Pi_{4}^{+}$ is given by $\varepsilon\sqrt{\bar{x}_{2}}$. This can be checked using the expressions of the transversal sections.}    
\end{remark}

With completely analogous constructions, one can define the transversal sections $\sigma_{i}^{-}$, $i = 1,\dots,4$ on the left-hand side of the Figure \ref{fig-transversal-sections}(a). One also can state completely analogous statements for the transition maps $\Pi_{i}^{-}:\sigma_{i-1}^{-}\rightarrow\sigma_{i}^{-}$ as those given by the previous Lemmas (in backward time). 

\subsection{Right and left-hand transition maps and difference map} Now, we are able to define the \emph{right} and \emph{left-hand transition maps} $\Pi_{\widetilde\mu}^{\pm}:\sigma_{0}\rightarrow\sigma_{5}$ by
$$\Pi_{\widetilde\mu}^{\pm}(y,\varepsilon,\widetilde\alpha) = \Pi_{5}^{\pm}\circ\Pi_{4}^{\pm}\circ\Pi_{3}^{\pm}\circ\Pi_{2}^{\pm}\circ\Pi_{1}^{\pm}(y,\varepsilon,\widetilde\alpha,\widetilde\mu),$$
whose properties are given in the next proposition.

In what follows, we define the integrals
\begin{equation}\label{eq-integrals-proof}
 J_{\tilde \mu}^{+}(y):=I_{5}^{+}(\xi_{X}(y), \tilde\mu) < 0, \quad \text{and} \quad J_{\tilde \mu}^{-}(y):=I_{5}^{-}(\xi_{Y}(y),\tilde\mu) < 0,   
\end{equation}
where $\xi_{X}$ and $\xi_{Y}$ are the half return maps defined in Section \ref{sec-candidates} and $I^{\pm}_{5}$ are defined in \eqref{eq-sdi-i5}.

\begin{remark}\label{remark-derivative}
{\rm Note that the derivative of $J_{\widetilde\mu}^{\pm}$ with respect to the variable $y$ is positive due to the chosen parametrization of $\sigma_{0}$.}   
\end{remark}

\begin{proposition}\label{prop-transitions}
The \emph{right} and \emph{left-hand transition maps} $\Pi_{\widetilde\mu}^{\pm}:\sigma_{0}\rightarrow\sigma_{5}$ are given by
\begin{equation}\label{eq-prop-full-transition}
\Pi_{\widetilde\mu}^{\pm}(y,\varepsilon,\widetilde\alpha) = f_{\widetilde\mu}^{\pm}(\varepsilon,\widetilde\alpha) - \operatorname{exp}\Bigg{(}\displaystyle\frac{J_{\widetilde\mu}^{\pm}(y) + o^{\pm}(1)}{\varepsilon^{2}}\Bigg{)},    \end{equation}
in which $f_{\widetilde\mu}^{\pm}$ are smooth functions satisfying $(f_{\widetilde\mu}^{+} - f_{\widetilde\mu}^{-})(0,0) = 0$ and $\frac{\partial (f_{\widetilde\mu}^{+} - f_{\widetilde\mu}^{-})}{\partial \widetilde\alpha}(0,0) \neq 0$. In addition, $J^{\pm}_{\widetilde\mu}$ are given by Equation \eqref{eq-integrals-proof}. Finally, $o^{\pm}(1)$ tends to zero uniformly as $\varepsilon \rightarrow 0$.
\end{proposition}

\begin{proof}

Using Lemmas \ref{lemma-p1+}--\ref{lemma-p24+} and similar results in backward time, we see that the $y$ component of the transition maps $\Pi_{\widetilde\mu}^{\pm}$ can be written as \eqref{eq-prop-full-transition}. The properties of $f_{\widetilde\mu}^{+} - f_{\widetilde\mu}^{-}$ follow from \cite[Theorem 4]{DeMaesschalckDumortier1} ($\widetilde\alpha$ is a regular breaking parameter and $\sigma_{5}$ is parameterized by the rescaled variable $\hat y$).
\end{proof}

\begin{remark}
{\rm In fact, the smooth function $f_{\widetilde\mu}^{+}$ in Equation \eqref{eq-prop-full-transition} is the smooth function $f_{5}^{+}$ given in Lemma \ref{lemma-p5+}. From now on, we work with the notation $f_{\widetilde\mu}^{+}$ for simplicity sake. The same remark on the notation of $f_{\widetilde\mu}^{-}$ and $f_{5}^{-}$ holds.}
\end{remark}

\begin{remark}
 {\rm Note that for $\varepsilon = 0$ the composition $\Pi_{4}^{\pm}\circ\Pi_{3}^{\pm}\circ\Pi_{2}^{\pm}\circ\Pi_{1}^{\pm}(y,\varepsilon,\widetilde\alpha,\widetilde\mu)$ in forward (resp. backward) time is given by the half return map $\xi_{X}$ (resp. $\xi_{Y}$).}    
\end{remark}

Define the \emph{difference map} $\Pi_{\widetilde\mu}$ as $\Pi_{\widetilde\mu} = \Pi_{\widetilde\mu}^{+} - \Pi_{\widetilde\mu}^{-}$. It follows from the construction that zeros of the difference map for $\varepsilon > 0$ correspond to periodic orbits of $\widetilde{Z}^{\varphi}_{\varepsilon,\widetilde\alpha,\widetilde\mu}$. In what follows, we prove that zeros of $\Pi_{\widetilde\mu}$ are related to zeros of the slow divergence integral $J_{\widetilde\mu}(y) = (J_{\widetilde\mu}^{+} - J_{\widetilde\mu}^{-})(y)$. In summary, we prove that the slow divergence integral can indeed be applied in order to detect limit cycles of $\widetilde{Z}^{\varphi}_{\varepsilon,\widetilde\alpha,\widetilde\mu}$.

\begin{mtheorem}\label{theo-main}
\emph{\textbf{(Terminal case)}} Denote $J_{\widetilde\mu}(y) = (J_{\widetilde\mu}^{+} - J_{\widetilde\mu}^{-})(y)$.
\begin{itemize}
    \item[\textbf{(a)}] Suppose that $\widetilde{y}\in\sigma_{0}$ is a zero of multiplicity $k$ of $J_{\widetilde\mu_{0}}(y)$. Then $\widetilde{Z}_{\varepsilon,\tilde\alpha,\widetilde\mu}^{\varphi}$ has at most $k+1$ limit cycles for $\varepsilon > 0$ sufficiently small, $\widetilde{\alpha}$ close to $0$ and $\widetilde{\mu}$ kept near $\widetilde{\mu}_{0}$, which are Hausdorff close to $\Upsilon^{\widetilde{y}}$.  
    \item[\textbf{(b)}] Assume that $J_{\widetilde\mu_0}(y)$ has exactly $k$ simple zeros $y_{1} <\dots < y_{k}$ in $\sigma_{0}$. Let $y_{0}\in\sigma_{0}$ satisfying $y_{0} < y_{1}$. Then there is a smooth function $\widetilde\alpha = \widetilde\alpha(\varepsilon,\widetilde\mu)$ satisfying $\widetilde\alpha(0,\widetilde\mu) = 0$ such that $\tilde{Z}_{\varepsilon,\widetilde\alpha(\varepsilon,\widetilde\mu),\widetilde\mu}^{\varphi}$ has $k+1$ periodic orbits $\Upsilon_{\varepsilon,\widetilde\mu}^{y_{i}}$, with $i=0,\dots,k$, for each $\varepsilon > 0$ sufficiently small. In addition, for $i = 0,\dots,k$, the periodic orbit $\Upsilon_{\varepsilon,\widetilde\mu}^{y_{i}}$ is isolated, hyperbolic and Hausdorff close to the canard cycle $\Upsilon^{y_{i}}$.
\end{itemize}

\end{mtheorem}

\begin{proof}
\textbf{Item (a).} By Rolle's Theorem, if $\Pi_{\widetilde\mu}'(y,\varepsilon,\widetilde\alpha)$ has $k$ zeros counting multiplicity, then $\Pi_{\widetilde\mu}(y,\varepsilon,\widetilde\alpha)$ has at most $k+1$ zeros. In other words, if $\Pi_{\widetilde\mu}'(y,\varepsilon,\widetilde\alpha)$ has $k$ zeros, then $\widetilde{Z}^{\varphi}_{\varepsilon,\widetilde\alpha,\widetilde\mu}$ has at most $k+1$ limit cycles. In what follows we work with $\Pi_{\widetilde\mu}'(y)$ (the prime $'$ denotes the derivative with respect to $y$).

Using Proposition \ref{prop-transitions}, one has that
$$\operatorname{exp}\Bigg{(}\displaystyle\frac{J_{\widetilde\mu}^{-}(y) + o^{-}(1)}{\varepsilon^{2}}\Bigg{)}\frac{1}{\varepsilon^{2}}\frac{\partial}{\partial y}\Big{(}J_{\widetilde\mu}^{-} + o^{-}(1)\Big{)} - \operatorname{exp}\Bigg{(}\displaystyle\frac{J_{\widetilde\mu}^{+}(y) + o^{+}(1)}{\varepsilon^{2}}\Bigg{)}\frac{1}{\varepsilon^{2}}\frac{\partial}{\partial y}\Big{(}J_{\widetilde\mu}^{+} + o^{+}(1)\Big{)}.$$

Now, using Remark \ref{remark-derivative}, we can rewrite
$$
\frac{1}{\varepsilon^{2}}\frac{\partial}{\partial y}\Big{(}J_{\widetilde\mu}^{\pm} + o^{\pm}(1)\Big{)} = \operatorname{exp}\Bigg{(}\frac{\varepsilon^{2}\ln(\frac{1}{\varepsilon^{2}}\frac{\partial}{\partial y}(J_{\widetilde\mu}^{\pm} + o^{\pm}(1)))}{\varepsilon^{2}}\Bigg{)} = \operatorname{exp}\Bigg{(}\frac{o^{\pm}(1)}{\varepsilon^{2}}\Bigg{)},
$$
and therefore
$$\Pi_{\widetilde\mu}'(y,\varepsilon,\widetilde{\alpha}) = \operatorname{exp}\Bigg{(}\displaystyle\frac{J_{\widetilde\mu}^{-}(y) + o^{-}(1)}{\varepsilon^{2}}\Bigg{)} - \operatorname{exp}\Bigg{(}\displaystyle\frac{J_{\widetilde\mu}^{+}(y) + o^{+}(1)}{\varepsilon^{2}}\Bigg{)},$$
where $o(1)^\pm$ tend to zero as $\varepsilon$ tends to zero, uniformly in $y$, $\widetilde\mu$ and $\widetilde\alpha$.

It follows that, for $\varepsilon > 0$, $\Pi_{\widetilde\mu}'(y,\varepsilon,\widetilde{\alpha}) = 0$ if, and only if, 
\begin{equation}\label{eq-proof-roots}
J_{\widetilde\mu}^{+}(y) - J_{\widetilde\mu}^{-}(y) + o(1) = 0,
\end{equation}
in which $o(1) = o^{+}(1) - o^{-}(1)$. Since we assume that $\tilde y$ is a zero of multiplicity $k$ of $J_{\widetilde\mu_0}$, then Equation \eqref{eq-proof-roots} has at most $k$ zeros (counting multiplicity) near $y=\tilde y$, for each $(\varepsilon, \widetilde{\alpha}, \widetilde{\mu})$ kept  near $(0, 0,\widetilde{\mu}_{0})$. This implies that $\Pi'$ has at most $k$ zeros (counting multiplicity) near $y=\tilde y$, for each $(\varepsilon, \widetilde{\alpha}, \widetilde{\mu})$ kept  near $(0, 0,\widetilde{\mu}_{0})$ and $\varepsilon > 0$. This completes the proof of  Item (a).

\textbf{Item (b).} From Proposition \ref{prop-transitions} and the Implicit Function Theorem, it follows that there is a smooth function $\widetilde\alpha(\varepsilon,\widetilde\mu)$ such that $\widetilde\alpha(0,\widetilde\mu) = 0$ and $(f_{\widetilde\mu}^{+} - f_{\widetilde\mu}^{-})\big{(}\varepsilon,\widetilde\alpha(\varepsilon,\widetilde\mu)\big{)} = 0$ for $\varepsilon \geq 0$ sufficiently small. One can rewrite the difference map $\Pi_{\widetilde\mu}$ as
$$\Pi_{\widetilde\mu}\big{(}y,\varepsilon,\widetilde\alpha(\varepsilon,\widetilde\mu)\big{)} = \operatorname{exp}\Bigg{(}\displaystyle\frac{J_{\widetilde\mu}^{-}(y) + o^{-}(1)}{\varepsilon^{2}}\Bigg{)} - \operatorname{exp}\Bigg{(}\displaystyle\frac{J_{\widetilde\mu}^{+}(y) + o^{+}(1)}{\varepsilon^{2}}\Bigg{)},$$
and $o^{\pm}(1)\rightarrow 0$ as $\varepsilon\rightarrow 0$, uniformly in $y$ and $\widetilde{\mu}$. Therefore, zeros of $\Pi_{\widetilde\mu}\big{(}y,\varepsilon,\widetilde\alpha(\varepsilon,\widetilde\mu)\big{)}$ with respect to the $y$ variable are given by 
$$J_{\widetilde\mu}^{-}(y) - J_{\widetilde\mu}^{+}(y) + o(1) = 0,$$
with $o(1) = o^{-}(1) - o^{+}(1)$ and obviously $o(1)\rightarrow 0$ as $\varepsilon\rightarrow 0$, uniformly in $y$ and $\widetilde{\mu}$. Suppose that $y_{1},\dots,y_{k}\in\sigma_{0}$ are simple zeros of $J_{\widetilde\mu_0}(y)$, and write $y_{1} < \dots <y_{k}$. For $\varepsilon > 0$ sufficiently small, those simple zeros persist, and therefore we have $k$ simple zeros of the difference map $\Pi_{\widetilde\mu}\big{(}y,\varepsilon,\widetilde\alpha(\varepsilon,\widetilde\mu)\big{)}$. Finally, we conclude that $\widetilde{Z}^{\varphi}_{\varepsilon,\widetilde\alpha(\varepsilon,\widetilde\mu),\widetilde\mu}$ has $k$ hyperbolic limit cycles $\Upsilon_{\varepsilon,\widetilde\mu}^{y_{i}}$, each one Hausdorff close to the canard cycle $\Upsilon^{y_{i}}$, for $i = 1,\dots,k$.

Fix $y_{0}\in\sigma_{0}$ such that $y_{0} < y_{1}$. Using suitable functions $f_{\tilde\mu}^{\pm}$ (see \cite{DeMaesschalckHuzak, Dumortier} for more details), it is possible to construct one extra hyperbolic limit cycle $\Upsilon_{\varepsilon,\widetilde\mu}^{y_{0}}$, and it will be the smallest limit cycle among $\Upsilon_{\varepsilon,\widetilde\mu}^{y_{i}}$, for $i = 0,\dots,k$.
\end{proof}

\begin{remark}\label{remark-limit-cycles}
{\rm In \cite{DeMaesschalckHuzak,Dumortier}, a result similar to Theorem \ref{theo-main}(b) has been proven for red canard cycles in Figure \ref{fig-intro-crossing}(b). More precisely, in our model \eqref{eq-def-regularization-2} such canard cycles can be parameterized by $y\in (0,y^{*})$ where we suppose that  the slow dynamics \eqref{eq-blow-up-e1-slow} along the portion of the critical curve $C_0$ below the line $y=y^{*}$ is regular (that is, negative). We define the following slow divergence integrals: $\overline{I}_{\widetilde\mu}^+(y) < 0$ (resp. $\overline{I}_{\widetilde\mu}^-(y) < 0$), equal to the integral $I_{\tilde\mu}^+(x_2)$ (resp. $I_{\tilde\mu}^-(x_2)$) defined in \eqref{PWSSDI}, where $x_2$ is the $x_2$-coordinate of the $\omega$-limit point (resp. $\alpha$-limit point) of the fast orbit of \eqref{eq-blow-up-e1-fast-sub} through $(x_2,y)=(0,y)$. Now, if $\overline{I}_{\widetilde\mu_{0}}^{+}(y) - \overline{I}_{\widetilde\mu_{0}}^{-}(y)$ has exactly $k$ simple zeros in the open interval $(0,y^*)$, we can produce $k+1$ limit cycles like in our Theorem \ref{theo-main}(b). An important difference lies in the fact that the extra limit cycle found in \cite{DeMaesschalckHuzak,Dumortier} is the biggest, that is, the extra limit cycle surrounds the $k$ limit cycles obtained from the simple zeros. On the other hand, in our Theorem \ref{theo-main}(b), the extra limit cycle is the smallest one.}

\end{remark}

\subsection{Proof of Theorem \ref{thm-center}}\label{application-center}


We consider a PSVF presenting a non-smooth center at $(0,1)\in\Sigma$ given by  
\begin{equation}\label{eq-II2-ns-center}
Z_{\alpha}(x,y) = \left\{
  \begin{array}{rcl}
   X_{\alpha}(x,y) & = & \big{(}y - 1, \alpha - 1\big{)}, \\
   Y_{\alpha}(x,y) & = & \big{(}y - 1, \alpha + 1 \big{)}.
  \end{array}
\right.
\end{equation}

Let $\widetilde{Z}_{\alpha}$ be a continuous combination of \eqref{eq-II2-ns-center} given by
\begin{equation}\label{eq-II2-center-cc}
\widetilde{Z}_{\alpha}(\lambda,x,y) = \left\{
  \begin{array}{rcl}
   \widetilde{Z}_{1}(\lambda,y)  & = & y - \lambda^{2}, \\
   \widetilde{Z}_{2,\alpha}(\lambda)  & = & \alpha - \lambda.
  \end{array}
\right.
\end{equation}

Observe that the continuous combination \eqref{eq-II2-center-cc} is a special case of \eqref{eq-combination-hopf} with $\mu=\alpha$ (we do not need the additional parameter $\widetilde\mu$) and $A_{\mu} = B_{\mu} \equiv 0$. The associated functions $F$ and $G$ (defined in \eqref{eq-def-f-g}) are given by $F(x) = \varphi^{2}(x)$ and $ G(x) = -\varphi(x)$. The critical curve $C_0$ is equal to $\{y=F(x_2)\}$ (see Section \ref{subsec-blow-up-e1}). 
We assume that $\varphi(0) = 0$ and $\varphi'(0) > 0$ (see the assumption \textbf{(A0)}). Moreover, in this section we assume that transition functions are monotonic. Then the assumptions \textbf{(A2)} and \textbf{(A3)} are satisfied on the interval $(-1,1)$. Since $A_{\mu} \equiv 0$, then assumption \textbf{(A1)} is  satisfied.

We focus on limit cycles of the following non-linear regularization (see \eqref{eq-def-regularization-2}):
$$\widetilde{Z}^{\varphi}_{\varepsilon,\widetilde{\alpha}}(x,y) := \widetilde{Z}_{\varepsilon\widetilde{\alpha}}\Big{(}\varphi\Big{(}\frac{x}{\varepsilon^{2}}\Big{)},x,y\Big{)}.$$

In addition, the half return maps of $Z_{0}$ in \eqref{eq-II2-ns-center} are given by $\xi_{X,Y}(y) = 2-y$ and $\xi(y) = y$ (recall that $\xi_{Y}$ is considered in backward time, see also Figure \ref{fig-half-return}). Note that the canard cycle $\Upsilon^{y}$ (see Figure \ref{fig-limit-cycles}(a)) is well-defined for each $y\in(1,2)$. We therefore restrict the half return maps to the interval $(1,2)$.

Our goal is to show that for any integer $k > 0$ there exists a monotonic function $\varphi_{k}$ such that the slow divergence integral $J$ from Theorem \ref{theo-main}(b) has $k$ simple zeros $1 < y_{1} <\dots < y_{k} < 2$. 

Let us first find the expression for $J$. From \eqref{eq-sdi-i5} and \eqref{eq-integrals-proof} it follows that
$$J(y) = I^{+}(x_2^{+}(2-y)) - I^{-}(x_2^{-}(2-y)),$$ where $I^{\pm}$ are defined in \eqref{PWSSDI} (recall that we do not have parameter $\widetilde\mu$). If we use the change of variable $x = x_2^{+}(2-y)$ and write $J(y)$ as $I(x)$, then we have
\begin{equation}\label{eq-II2-center-integral}
I(x) = I^{+}(x)-I^{-}(L(x)) = 4\displaystyle\int_{x}^{L(x)} \varphi(s)\big{(}\varphi'(s)\big{)}^{2}ds, 
\end{equation}
where $x \in (0,1)$ and $L(x) < 0$ satisfies $\varphi^{2}(x)=\varphi^{2}(L(x))$, due to $\xi(y) = y$. Now, it suffices to prove that for any integer $k > 0$ there is a monotonic function $\varphi_{k}$ such that  $I$  has $k$ (positive) simple zeros.

Consider a function $\widetilde{\varphi}(x) = x + \delta\varphi_{e}(x)$, in which $\varphi_{e}$ is an even polynomial satisfying $\varphi_{e}(0) = 0$. We fix a compact interval $[-2\nu,2\nu]\subset (-1,1)$, with a small $\nu > 0$. The function $\widetilde\varphi$ is monotonic in $[-2\nu,2\nu]$ for any $\delta \geq 0$ sufficiently small. Using the definition of $\widetilde\varphi$ and $\widetilde\varphi^{2}(x) = \widetilde\varphi^{2}(L(x))$, one can also write $L(x) = -x + L_{1}(x)\delta + O(\delta^{2})$, where $L_{1}(x) = -2\varphi_{e}(x)$.


Therefore, for $x\in (0,\nu]$ and $\delta \geq 0$ sufficiently small, we have
\begin{equation*}
\begin{array}{rcl}
    \displaystyle\frac{1}{4}I(x) & = &  \displaystyle\int_{x}^{L(x)}\widetilde{\varphi}(s)\big{(}\widetilde{\varphi}'(s)\big{)}^{2}ds = \displaystyle\int_{x}^{L(x)}\big{(}s + \delta\varphi_{e}(s)\big{)}\big{(}1 + \delta\varphi_{e}'(s)\big{)}^{2}ds \\
     & = & \displaystyle\int_{x}^{L(x)}sds + \delta\displaystyle\int_{x}^{L(x)}\big{(}2s\varphi_{e}'(s) + \varphi_{e}(s)\big{)}ds + O(\delta^{2}) \\
     & = & -\delta xL_{1}(x) + \delta\displaystyle\int_{x}^{-x}\big{(}2s\varphi_{e}'(s) + \varphi_{e}(s)\big{)}ds + O(\delta^{2}).
\end{array}    
\end{equation*}

Observe that $2s\varphi_{e}'(s) + \varphi_{e}(s)$ is an even polynomial. Thus
$$
\displaystyle\int_{x}^{-x}\big{(}2s\varphi_{e}'(s) + \varphi_{e}(s)\big{)}ds = - 2\displaystyle\int_{0}^{x}\big{(}2s\varphi_{e}'(s) + \varphi_{e}(s)\big{)}ds,
$$
because we are integrating in a symmetric interval. We also remark that
$$
\displaystyle\int_{0}^{x}s\varphi_{e}'(s)ds = x\varphi_{e}(x)  - \displaystyle\int_{0}^{x}\varphi_{e}(s)ds,
$$
and then we finally have
\begin{equation}\label{eq-teo-c-integral}
\displaystyle\frac{1}{4}I(x) = -2\delta\displaystyle\int_{0}^{x}s\varphi_{e}'(s)ds + O(\delta^{2}).
\end{equation}

We conclude that simple zeros of the integral in the right hand side of \eqref{eq-teo-c-integral} will persist as simple zeros of $I(x)$, for each small but positive $\delta$. 

Given any positive integer $k$, take $a_{1} <\dots < a_{k}$ in the open interval $(0,\nu)$. Define the odd polynomial $P(x) = x^{3}(x^{2} - a_{1}^{2})\dots(x^{2} - a_{k}^{2})$. The even polynomial $\varphi_{e}(x) = \int_{0}^{x}\frac{P'(s)}{s}ds$ satisfies $\int_{0}^{x}s\varphi_{e}'(s)ds= P(x)$ and therefore the integral in the right hand side of \eqref{eq-teo-c-integral} has $k$ simple zeros $a_{1} <\dots < a_{k}$ in $(0,\nu)$. They persist as simple zeros of $I(x)$ in $(0,\nu)$, for positive but small $\delta$. 


The next step is smoothly extend $\widetilde{\varphi}(x)$ to a monotonic transition function $\varphi_{k}$ in the interval $[-1,1]$.
Define a \emph{bump function} $\beta:\mathbb{R}\rightarrow\mathbb{R}$ with support $\operatorname{supp}(\beta) = (-2,2)$ and satisfying $\beta(x) \equiv 1$ in the compact interval $[-1,1]$. Now, define $\beta_{\nu}(x) = \beta\left(\frac{x}{\nu}\right)$. With this construction, we have $\operatorname{supp}(\beta_{\nu}) = (-2\nu,2\nu)$ and $\beta_{\nu}\equiv 1$ in the compact interval $[-\nu,\nu]$. Finally, using the monotonic transition function
$$\psi(x) = \left\{
  \begin{array}{ccc}
   \tanh{\left(\frac{x}{1-x^{2}}\right)} & \text{if} & |x| < 1, \\
    \operatorname{sgn}x& \text{if} & |x| \geq 1,
  \end{array}
\right.$$
we define $\varphi_{k}(x) = \widetilde{\varphi}(x)\beta_{\nu}(x) + \psi(x)\left(1-\beta_{\nu}(x)\right)$.

We claim that $\varphi_{k}$ is a monotonic transition function. Indeed, observe that $\psi$, $\beta_{\nu}$ and consequently $\varphi_{k}$ are $C^{\infty}$-smooth functions. Furthermore, $\varphi_{k} \equiv \widetilde{\varphi}$ and $\varphi_{k} \equiv \psi$ in $[-\nu,\nu]$ and $(-\infty,-2\nu]\cup[2\nu,\infty)$, respectively. Therefore, it remains to check that $\varphi_{k}$ is monotonic in $(-2\nu, -\nu)\cup(\nu,2\nu)$. In order to do this, we can proceed as in the proof of \cite[Theorem 4.3]{HuzakKristiansen} because $\widetilde{\varphi}(0) = \psi(0) = 0$ and $\widetilde{\varphi}'(0) = \psi'(0) = 1$.

Now Theorem \ref{thm-center} follows directly from Theorem \ref{theo-main}(b). The $k + 1$ limit cycles are produced by the canard cycles $\Upsilon^y$ (see Figure \ref{fig-limit-cycles}(a)). Bearing in mind Remark \ref{remark-limit-cycles}, these $k+1$ limit cycles can also be produced by red canard cycles in Figure \ref{fig-intro-crossing}(b). Indeed, we have the same integral $I(x)$ as above and we can construct $k$ simple zeros using the same steps (see also Section \ref{application-II2}). The $k+1$ red limit cycles (as in Figure \ref{fig-intro-crossing}(b)) occur for a new control curve $\widetilde\alpha = \widetilde\alpha(\varepsilon)$. $\Box$

\subsection{Proof of Theorem \ref{thm-ii2}}\label{application-II2}


Consider the linear PSVF (for $\alpha\sim 0)$
\begin{equation}\label{eq-cod1-ii-psvf}
Z_{\alpha}(x,y) = \left\{
  \begin{array}{ll}
   X_{\alpha}(x,y) =  \big{(}y - 1 - \alpha, & \alpha - 1 -x), \\
   Y_{\alpha}(x,y) =  \big{(}y - 1 -x , & \alpha + 1 -x).
     \end{array}
\right.
\end{equation}

Let us describe the phase portrait of $Z_{\alpha}$. The points $P_{\alpha} = (-1 + \alpha, 1 + \alpha)$ and $Q_{\alpha} = (1+\alpha,2+\alpha)$ are (virtual) linear center and attracting focus of $X_{\alpha}$ and $Y_{\alpha}$, respectively. It can also be checked that both $T_{\alpha}^{X} = (0,1+\alpha)$ and $T_{\alpha}^{Y} = (0,1)$ are invisible fold points. Geometrically, it is easy to see that, for $\alpha = 0$, the fold-fold singularity behaves like a non-smooth attracting focus, because $Y_{\alpha}$ has a smooth attracting focus. The segment limited by these fold points is the sliding region $\Sigma^{s}$, and the sewing region is given by $\Sigma^{w} = \Sigma \backslash \big{(}\Sigma^{s} \cup T_{\alpha}^{X} \cup T_{\alpha}^{Y}\big{)}$. The sliding vector field is given by $Z^{\Sigma}_{\alpha}(y) = \alpha^{2} + \alpha - 2y + 2$, whose equilibrium point $S_{\alpha} = \big{(}0, \frac{1}{2}(\alpha^{2} + \alpha + 2)\big{)}$ is repelling for $\alpha < 0$ and attracting for $\alpha > 0$. 


In order to verify that the invisible-invisible fold-fold $(0,1)\in\Sigma$ has codimension $1$, we must compute $\xi_{X,Y}$. One can explicitly solve the system of ODEs associated to $X_{0}$ and obtain $\xi_{X}(y) = 2 - y$. Using Taylor series of order $2$ of the solutions of $Y_{0}$, we obtain
$$\xi_{Y}^{-1}(y) = \frac{-y^{3} + 4y^{2} -8y + 6}{(y-2)^{2}}.$$

Therefore
\begin{equation*}
\begin{array}{rcl}
    \xi(y) & = & \xi_{Y}^{-1}\circ\xi_{X}(y)  = - \frac{2}{y^2} + \frac{4}{y} + y - 2\\
     & = & 1 + (y-1) - 2(y-1)^{2} + 4(y-1)^{3} - 6 (y-1)^{4} + O\left((y-1)^5\right).
\end{array}    
\end{equation*}

The coefficient of $(y-1)^{2}$ is nonzero, therefore this singularity has codimension 1. Moreover, such coefficient is negative, therefore this II$_{2}$ has an attracting focus-like behavior.

A continuous combinations of \eqref{eq-cod1-ii-psvf} is given by
\begin{equation}\label{eq-cod1-ii-combination}
\widetilde{Z}_{\alpha}(\lambda,x,y) = \Big{(}y - \lambda^{2} - \frac{(\alpha - x)\lambda^{m-1}}{2} - \frac{(\alpha + x)\lambda^{m}}{2}, \quad \alpha - \lambda - x\lambda^{n}\Big{)},
\end{equation}
with $m\geq 4$ and $n\geq 2$ being even integers. The continuous combination \eqref{eq-cod1-ii-combination} is a special case of \eqref{eq-combination-hopf} with $\mu =\alpha$, $A_{\alpha}(\lambda,x) = - \frac{(\alpha - x)\lambda^{m-1}}{2} - \frac{(\alpha + x)\lambda^{m}}{2}$ and $B_{\alpha}(\lambda,x,y) = - x\lambda^{n}$. Using \eqref{eq-def-f-g}, one obtains $F(x)=\varphi^2(x)$ and $G(x)=-\varphi(x)$. The critical curve $C_{0}$ is given by $\{y = F(x)\}$. Assuming that $\varphi$ is monotonic and $\varphi(0) = 0$, then the assumptions \textbf{(A0)}-\textbf{(A3)} are satisfied. 

In what follows, we focus on red canard cycles in Figure \ref{fig-intro-crossing}(b) of the regularization \eqref{eq-def-regularization-2}. We believe that the same result is true for big limit cycles (see Figure \ref{fig-limit-cycles}(a)). This case is more technical and it is a topic of further study. 



Using Remark \ref{remark-limit-cycles} and parameterizing the slow divergence integrals by $x\in (0,1)$ instead of $y\in (0,1)$ (we use the change of variable $y = F(x)$), it suffices  to study simple zeros of
$$I(x) = 4\int_{x}^{L(x)}\varphi(s)\left(\varphi'(s)\right)^{2}ds,$$
where $x\in (0,1)$, $L(x) < 0$ and $F(x) = F(L(x))$. This slow divergence integral has the same form as \eqref{eq-II2-center-integral} in the non-smooth center case. Using the same reasoning as in Section \ref{application-center} we can show that for any integer $k > 0$ there is a monotonic
function $\varphi_{k}$ such that $I$ has $k$ (positive) simple zeros. Finally, Remark \ref{remark-limit-cycles} implies that $k+1$ hyperbolic limit cycles can be produced inside the regularization stripe. $\Box$







\section{Analysis of the dodging case}\label{section-dodging-case}

In this section we consider the dodging case defined in Section \ref{sec-candidates} (we suppose that the assumptions \textbf{(A0)}-\textbf{(A3)} are satisfied). In this case, the slow divergence integral can also be used in order to study limit cycles of the regularized
vector field defined in \eqref{eq-def-regularization-2}, produced by canard cycles $\Gamma^{\widetilde{y}}$ in Figure \ref{fig-limit-cycles}(b).

Of course, one must define transversal sections and transition maps for this case. We will not go into details since the construction, lemmas and propositions can be done in an analogous way as it was done in Section \ref{sec-terminal-case}. We refer to Figure \ref{fig-transversal-sections}(b) for the transversal sections $\sigma_{i}^{-}$ for $i=0,\dots,5$. We use the same parameterization of $\sigma_{0,5}$ as in Section \ref{sec-sections-transitions}.

The transition maps $\Pi_{\widetilde{\mu}}^{\pm}: \sigma_{0}\rightarrow\sigma_{5}$ are given by
\begin{equation}\label{eq-transition-dodging}
    \begin{array}{rcl}
         \Pi_{\widetilde{\mu}}^{+}(y,\varepsilon,\widetilde{\alpha}) & = & f_{\widetilde{\mu}}^{+}(\varepsilon,\widetilde{\alpha}) + \operatorname{exp}\Bigg{(}\displaystyle\frac{\overline{I}_{\widetilde{\mu}}^{+}(y) + o^{+}(1)}{\varepsilon^{2}}\Bigg{)},  \\
         \Pi_{\widetilde{\mu}}^{-}(y,\varepsilon,\widetilde{\alpha}) & = & f_{\widetilde{\mu}}^{-}(\varepsilon,\widetilde{\alpha}) -\operatorname{exp}\Bigg{(}\displaystyle\frac{J_{\widetilde{\mu}}^{-}(y) + o^{-}(1)}{\varepsilon^{2}}\Bigg{)}, 
    \end{array}
\end{equation}
in which the smooth functions $f_{\widetilde{\mu}}^{\pm}$ and $o^{\pm}(1)$ satisfy analogous properties to Proposition \ref{prop-transitions},  $J_{\widetilde\mu}^{-}(y)$ is defined in \eqref{eq-integrals-proof} and $\overline{I}_{\widetilde\mu}^{+}(y)$ in Remark \ref{remark-limit-cycles}. The expression for the backward transition map $\Pi_{\widetilde\mu}^{-}$ follows from \eqref{eq-prop-full-transition}, while the expression for the forward transition map $\Pi_{\widetilde\mu}^{+}$, together with the properties of $f_{\widetilde\mu}^{+} - f_{\widetilde\mu}^{-}$ stated in Proposition \ref{prop-transitions}, follows from \cite[Theorem 4]{DeMaesschalckDumortier1}. Observe that we have a plus sign in front of the exponential in the expression of $\Pi_{\widetilde{\mu}}^{+}$. The difference map is then given by $\Pi_{\widetilde{\mu}} = \Pi_{\widetilde{\mu}}^{+} - \Pi_{\widetilde{\mu}}^{-}$ and the derivative of $\overline{I}_{\widetilde\mu}^{+}(y)$ with respect to $y$ is negative, due to the chosen parameterization of $\sigma_0$. See also Remark \ref{remark-derivative}.

We have the following result.

\begin{mtheorem}\label{theo-dodging}
\emph{\textbf{(Dodging case)}} Denote $I_{\widetilde{\mu}}(y) = (\overline{I}_{\widetilde{\mu}}^{+} - J_{\widetilde{\mu}}^{-})(y)$, and let $\widetilde{y}\in\sigma_{0}$. Then $\widetilde{Z}^{\varphi}_{\varepsilon,\widetilde\alpha,\widetilde\mu}$ has at most $2$ limit cycles produced by the canard cycle $\Gamma^{\widetilde{y}}$, for $\varepsilon > 0$ sufficiently small, $\widetilde{\alpha}$ close to zero and $\widetilde{\mu}$ close to $\widetilde{\mu}_{0}$. Moreover, if $y = \widetilde{y}$ is a simple root of $I_{\widetilde{\mu}_{0}}(y)$, then, for each $\varepsilon > 0$ sufficiently small and $\widetilde{\mu}$ close to $\widetilde{\mu}_{0}$, the $\widetilde{\alpha}$-family $\widetilde{Z}^{\varphi}_{\varepsilon,\widetilde\alpha,\widetilde\mu}$ undergoes a saddle-node bifurcation of limit cycles, which are Hausdorff close to $\Gamma^{\widetilde{y}}$.
\end{mtheorem}
\begin{proof}
Firstly, we prove the existence of at most two limit cycles. With a similar reasoning as in the proof of Theorem \ref{theo-main}(a), one can show that, for each small $\varepsilon > 0$,   $\frac{\partial \Pi_{\widetilde{\mu}}}{\partial y}(y) = 0$ is equivalent to
\begin{equation}\label{eq-proof-dodging}
\overline{I}_{\widetilde{\mu}}^{+}(y) - J_{\widetilde{\mu}}^{-}(y) + o(1) = 0,    
\end{equation}
where $o(1) \rightarrow 0$ as $\varepsilon \rightarrow 0$. Notice that the derivative of the left hand side of \eqref{eq-proof-dodging} with respect to $y$ is negative for each $\varepsilon\geq 0$ small enough (we use Remark \ref{remark-derivative} and the fact that the derivative of $\overline{I}_{\widetilde{\mu}}^{+}$ is negative).  Rolle's Theorem implies that \eqref{eq-proof-dodging} has at most $1$ solution (counting multiplicity) near $y = \widetilde{y}$.  Thus, $\frac{\partial \Pi_{\widetilde{\mu}}}{\partial y}(y)$ has at most $1$ zero near $y = \widetilde{y}$. Applying Rolle's Theorem once more, we conclude that $\Gamma^{\widetilde{y}}$ can produce at most $2$ limit cycles. This completes the proof of the first part of Theorem \ref{theo-dodging}. The second part of Theorem \ref{theo-dodging} can be proved in the same fashion as Theorem 4.3(3) in \cite{Dumortier}.

\end{proof}

\section{Acknowledgements}

Peter De Maesschalck is supported by Flanders FWO agency (G0F1822N grant). Renato Huzak is supported by Croatian Science Foundation (HRZZ) grant IP-2022-10-9820. Otavio Henrique Perez is supported by Sao Paulo Research Foundation (FAPESP) grants 2021/10198-9 and 2024/00392-0. 

\section{Conflict of interest}

On behalf of all authors, the corresponding author states that there is no conflict of interest.

\end{document}